\documentclass[11pt, letterpaper, onecolumn]{IEEEtran}

\usepackage[T1]{fontenc}
\usepackage{lmodern}
\usepackage[utf8]{inputenc}
\usepackage[numbers,compress]{natbib}
\usepackage{url}
\usepackage[english]{babel}
\usepackage{tikz}
\usetikzlibrary{shapes,arrows}
\tikzstyle{measurement} = [draw, fill=white!20, rectangle, 
minimum height=3em, minimum width=3em]
\tikzstyle{block} = [draw, fill=white!20, rectangle, 
minimum height=2em, minimum width=2em]
\tikzstyle{block1} = [draw, fill=white!20, rectangle, 
minimum height=2em, minimum width=2em]
\tikzstyle{sum} = [draw, fill=blue!20, circle, node distance=1cm]
\tikzstyle{input} = [coordinate]
\tikzstyle{output} = [coordinate]
\tikzstyle{pinstyle} = [pin edge={to-,thin,black}]

\usepackage{verbatim}
\usepackage{amsthm}
\usepackage{amsmath}
\usepackage{algcompatible}
\usepackage{algorithm}
\usepackage[noend]{algpseudocode}
\usepackage{booktabs}
\usepackage{enumerate}
\usepackage{graphics}
\usepackage{amssymb}
\usepackage{latexsym}
\usepackage{bbm}
\usepackage{needspace}
\usepackage{color, colortbl}
\usepackage[english]{babel}
\usepackage{tikz}
\usepackage{caption}
\usepackage{subcaption}
\usetikzlibrary{arrows,automata}
\usetikzlibrary{positioning}
\usepackage{filecontents}
\usepackage{mathtools}
\usepackage{multirow}
\usepackage{booktabs}

\usepackage[margin=1.0in]{geometry}

\theoremstyle{plain}
\newtheorem{thm}{\protect\theoremname}
\theoremstyle{plain}

\newtheorem{remark}{Remark}

\newtheorem{problem}{Problem}

\newtheorem{proposition}{Proposition}

\makeatother

\providecommand{\lemmaname}{Lemma}
\providecommand{\theoremname}{Theorem}

\newcommand{\epsA}{\epsilon_A}
\newcommand{\epsB}{\epsilon_B}
\newcommand{\indinf}{ {\infty \rightarrow \infty} }
\newcommand{\tp}{\top}
\newcommand{\xx}{\mathbf{x}}
\newcommand{\uu}{\mathbf{u}}
\newcommand{\ww}{\mathbf{w}}
\newcommand{\dsuper}[3]{#1^{#2,#3}}
\newcommand{\KK}{\mathbf{K}}
\newcommand{\sA}{\mathbf{A}}
\newcommand{\sB}{\mathbf{B}}
\newcommand{\sR}{\mathcal{R}}
\newcommand{\sQ}{\mathcal{Q}}

\newcommand{\Phix}{\mathbf{\Phi}_x}
\newcommand{\Phiu}{\mathbf{\Phi}_u}
\newcommand{\DDelta}{\mathbf{\Delta}}
\newcommand{\PPhi}{\mathbf{\Phi}}
\newcommand{\MM}{\mathbf{M}}
\newcommand{\DeltaVar}{\mathbf{\Delta}}
\newcommand{\tildeww}{{\tilde{\mathbf{w}}}}
\newcommand{\circled}[1]{\raisebox{.5pt}{\textcircled{\raisebox{-.9pt} {#1}}}}
\newcommand{\XX}{\mathbf{X}}
\newcommand{\YY}{\mathbf{Y}}

\allowdisplaybreaks

\title{{\LARGE \bf Robust Closed-loop Model Predictive Control \\ via System Level Synthesis}}
\author{Shaoru Chen$^\star$, Han Wang$^\star$, Manfred Morari, Victor M. Preciado,  Nikolai Matni
\thanks{Han Wang is with the Department of Applied Mathematics and Computational Science, University of Pennsylvania, Philadelphia, PA, 19104, USA (e-mail: wanghan2@sas.upenn.edu).
		
Shaoru Chen, Manfred Morari, Victor M. Preciado and Nikolai Matni are with the Department of Electrical and Systems Engineering, University of Pennsylvania, Philadelphia, PA, 19104, USA (e-mail: \{srchen, morari, preciado, nmatni\}@seas.upenn.edu).
		
$\star$The first two authors contributed equally to this paper.}} 

\date{}

\begin{document}
\pagestyle{plain}
\maketitle

\begin{abstract}
	
In this paper, we consider the robust closed-loop model predictive control (MPC) of a linear time-variant (LTV) system with norm bounded disturbances and LTV model uncertainty, wherein a series of constrained optimal control problems (OCPs) are solved. Guaranteeing robust feasibility of these OCPs is challenging due to disturbances perturbing the predicted states, and model uncertainty, both of which can render the closed-loop system unstable. As such, a trade-off between the numerical tractability and conservativeness of the solutions is often required. We use the System Level Synthesis (SLS) framework to reformulate these constrained OCPs over closed-loop system responses, and show that this allows us to transparently account for norm bounded additive disturbances and LTV model uncertainty by computing robust state feedback policies.  We further show that by exploiting the underlying linear fractional structure of the resulting robust OCPs, we can significantly reduce the conservativeness of existing SLS-based and tube-MPC-based robust control methods while also improving computational efficiency. We conclude with numerical examples demonstrating the effectiveness of our methods.

\end{abstract}

\section{Introduction}

Model predictive control (MPC) has achieved remarkable success in solving multivariable constrained control problems across a wide range of application areas, such as process control~\cite{qin2003survey}, automotive systems~\cite{hrovat2012development}, power networks~\cite{negenborn2007multi}, and robot locomotion~\cite{wieber2006trajectory}. In MPC, a control action is computed by solving a finite-horizon constrained optimal control problem (OCP) at each sampling time, and then applying the first control action. The stability and performance of MPC depends on the accuracy of the model being used, and indeed robustness to both additive disturbances and model uncertainty must be considered.  Although MPC using a nominal model (i.e., one ignoring uncertainty) offers some level of robustness~\cite{marruedo2002stability}, it has been shown that the closed-loop system achieved by nominal MPC can be destabilized by an arbitrarily small disturbance~\cite{grimm2004examples}. As a result, robust MPC, which explicitly deals with uncertainty, has received much attention~\cite{bemporad1999robust}.  

When only additive disturbances are present, open-loop robust MPC, which optimizes over a sequence of control actions $\mathbf{u} = \{ u_0, \cdots, u_{N-1} \}$ subject to suitably robust constraints, can be applied but tends to be overly conservative or even infeasible~\cite{mayne2000constrained}, as a single sequence of inputs $\mathbf{u}$ is chosen for all possible disturbance realizations. On the other hand, closed-loop MPC, which optimizes over the control policies $\pi = \{\pi_0(\cdot), \cdots, \pi_{N-1}(\cdot) \}$, can reduce the conservativeness of the solutions. However, the policy space is infinite-dimensional and renders the online OCPs intractable. The problem can be rendered tractable by restricting the policies $\pi$ to lie in a function class that admits a finite-dimensional parameterization.  For example,  
policies of the form $\pi_i(x) = Kx + v_i$ are considered in \cite{kouvaritakis2000efficient, schuurmans2000robust, lee2000robust}, where $K$ is a pre-stabilizing feedback gain $K$ that is fixed beforehand, thus reducing the decision variables to the vectors $\{v_1, \cdots, v_{N-1} \}$. To reduce conservativeness, an affine feedback control law $\pi(x) = K_i x + v_i$ can be applied with decision variables $K_i$ and $v_i$; however, the resulting OCP is non-convex.  In \cite{goulart2006optimization}, it was observed that by re-parameterizing the OCPs to be over disturbance based feedback policies of the form $\pi_i(w) = \sum_{j=0}^{i-1} M_{ij}w_j + v_i$, the resulting OCPs are convex. In~\cite{mayne2005robust,langson2004robust}, the authors propose an alternative (tube-MPC) invariant set based approach, which bounds the system trajectories within a tube that robustly satisfy constraints.

The more challenging problem considering model uncertainty is tackled in~\cite{kouvaritakis2000efficient, langson2004robust, kothare1996robust}. When polytopic or structured feedback model uncertainty occurs, a linear matrix inequality (LMI) based robust MPC method is proposed in~\cite{kothare1996robust}. When both model uncertainty and additive disturbances are present, the method proposed in~\cite{langson2004robust} designs tubes containing all possible trajectories under polytopic uncertainty assumptions. Alternative approaches based on dynamic programming (DP)~\cite{borrelli2017predictive} are shown to obtain tight solutions, but the computation quickly becomes intractable. Adaptive robust MPC, which considers estimation of the parametric uncertainty while implementing robust control, is proposed in~\cite{bujarbaruah2019semi, kim2018robust,lorenzen2019robust}.
\begin{table}
	\centering
	\footnotesize
	\caption{Comparison of the SLS, tube and DP based methods on robust OCP with norm-bounded uncertainty}
	\label{table:compare}
	\makebox[\textwidth]{\begin{tabular}{l|l|c|c|c}
		\toprule
		\multicolumn{2}{l|}{ } & Robust SLS  MPC & Tube MPC & Dynamic Programming \\ \midrule
		\multirow{2}{*}{LTV model uncertainty} & Memoryless & yes & yes & yes \\ 
		& With memory & yes & no & no \\ \midrule 
		\multicolumn{2}{l|}{Additive disturbance} & $\ell_\infty$ norm bounded & polytopic & polytopic \\ \midrule
		\multirow{2}{*}{Complexity} & No. of variables & $(T+1)^2n(n+m)$ & $T(1+n) + m +(T-1)mV$ & \multirow{2}{*}{exponential in problem size}\\ 
		& No. of constraints & $O(T(T+1)n^2/2)$  & $O(T(4mn)^n V H)$  & \\ \midrule
		\multicolumn{2}{l|}{Exactness} & See Section~\ref{sec:methodscomparison} & See Section~\ref{sec:methodscomparison} & exact \\
		\bottomrule
	\end{tabular}}
	\label{tab:multirow}
\end{table}

As described above, there is a rich body of work addressing the robust MPC problem, and it remains an active area of research for which no definitive solution exists. Due to the inherent intractability of the general robust MPC problem subject to both additive disturbance and model uncertainty, all of the aforementioned methods trade off conservativeness for computational tractability in different ways. The recently developed \emph{System Level Synthesis} (SLS) parameterization~\cite{ANDERSON2019364} provides an alternative approach to tackling the robust MPC problem and exploring this tradeoff space. The SLS framework transforms the OCP from one over control laws to one over closed-loop system responses, i.e., the linear maps taking the disturbance process to the states and inputs in closed loop.  Robust variants of the SLS parameterization allow for a transparent mapping of model uncertainty on system behavior, providing a transparent characterization of the joint effects of additive disturbances and model errors when solving OCPs~\cite{dean2019safely}. In this paper, we apply SLS to the robust MPC problem under both additive disturbance and LTV dynamic model uncertainty.\footnote{By dynamic model uncertainty, we mean that model perturbations may depend on past states as well. This will be formalized in Section~\ref{sec:formulation}} Our contributions are summarized below:
\begin{enumerate}
\item We propose and analyse an SLS based solution to the robust MPC problem for LTV systems subject to additive disturbances and LTV dynamic model uncertainty, and show that it allows a favorable tradeoff between conservativeness and computational complexity. 
\item We reduce the conservativeness of both existing SLS and tube MPC based robust control methods, while also reducing computational complexity, by exploiting the underlying linear fractional structure of the resulting robust OCPs. 
\end{enumerate}
Table~\ref{table:compare} summarizes the generality (as measured by the type of uncertainty the method can accommodate), computational complexity (as measured by the number of variables and constraints in the resulting OCP), and exactness of SLS, tube~\cite{langson2004robust}, and DP~\cite{borrelli2017predictive} based solutions to the robust MPC problem considered in this paper. The parameters $n, m, T, V, H$ denote the state dimension, input dimension, MPC horizon, the number of vertices of the tube, the number of inequalities to describe the polytopic tube, respectively. As can be seen, our approach compares favorably in terms of generality and computational complexity, and as we demonstrate in Section~\ref{sec:examples}, appears to outperform tube MPC in terms of conservativeness, especially over long horizons.

%

The remainder of the paper is organized as follows. Section~\ref{sec:formulation} states the robust MPC problem formulation. After introducing SLS preliminaries in Section~\ref{sec:SLSpreliminaries}, we develop the SLS formulation of robust MPC in Section~\ref{sec:SLSuncertainty}. Section~\ref{sec:examples} compares the proposed method with other robust MPC methods, and Section~\ref{sec:conclusion} concludes the paper.

\textbf{Notation}: 
$x_{i:j}$ is shorthand for the set $\{x_i, x_{i+1}, \cdots, x_{j} \}$. We use bracketed indices to denote the time of the real system and subscripts to denote the prediction time indices within an MPC loop, i.e., the system state at time $t$ is $x(t)$ and the $t$-th prediction is $x_t$. If not explicitly specified, the dimensions of matrices are compatible and can be inferred from the context. For two vectors $x$ and $y$, $x \leq y$ denotes element-wise comparison. For a symmetric matrix $Q$, $Q \succeq 0$ denotes that $Q$ is positive semidefinite. A linear, causal operator $\mathbf{R}$ defined over a horizon of $T$ is represented by 
\begin{equation} \label{eq:BLT}
\mathbf{R} = \begin{bmatrix}
\dsuper{R}{0}{0} & \ & \ & \ \\
\dsuper{R}{1}{1} & \dsuper{R}{1}{0} & \ & \ \\
\vdots & \ddots & \ddots & \ \\
\dsuper{R}{T}{T} & \cdots & \dsuper{R}{T}{1} & \dsuper{R}{T}{0}
\end{bmatrix}
\end{equation} 
where $\dsuper{R}{i}{j} \in \mathbb{R}^{p \times q}$ is a matrix of compatible dimension. We denote the set of such matrices by $\mathcal{L}_{TV}^{T, p \times q}$ and will drop the superscript $T, p \times q$ when it is clear from the context. $\mathbf{R}(i,:)$ denotes the $i$-th block row of $\mathbf{R}$ and $\mathbf{R}(:,j)$ denotes the $j$-th block column of $\mathbf{R}$, both indexing from $0$.

\section{Problem Formulation}
\label{sec:formulation}
In this paper, we consider robust MPC of the LTV system 
\begin{equation} \label{eq:RealSys}
	x(k+1) = A(k) x(k) + B(k) u(k) + w(k), k \geq 0,
\end{equation}
where $x(k) \in \mathbb{R}^n$ is the state, $u(k) \in \mathbb{R}^m$ is the control input, and $w(k) \in \mathbb{R}^n$ is the disturbance at time $k$. The matrices $(A(k), B(k))$ denote the real system matrices at time $k$, but are unknown. Instead, we assume that a nominal model  $(\hat{A}(k), \hat{B}(k))$ is available, and that the model errors are bounded in terms of the $\indinf$ induced norm as:
\begin{equation} \label{eq:UncertaintySet}
	\lVert \Delta_A(k) \rVert_\indinf \leq \epsA, \lVert \Delta_B(k) \rVert_\indinf \leq \epsB, \forall k \geq 0
\end{equation}
where $\Delta_A(k) = \hat{A}(k) - A(k), \Delta_B(k) = \hat{B}(k) - B(k)$. Further, we assume the disturbance is norm-bounded by:
\begin{equation} \label{eq:wSet}
	w(k) \in \mathcal{W} = \lbrace w \in \mathbb{R}^n \mid \lVert w \rVert_\infty \leq \sigma_w \rbrace.
\end{equation}

In robust MPC, a series of finite horizon robust constrained OCPs are solved, with the current state $x(k)$ set as the initial condition. Motivated by the norm-bounded uncertainty in~\eqref{eq:UncertaintySet}, in each MPC loop, we consider the following more general form of uncertainty described by linear, causal operators and formulate our problem as follows:

\begin{problem} \label{prob:robustConstrControl}
Consider the robust MPC of an LTV system over horizon $T$ with additive disturbances and LTV dynamic model uncertainty. For each OCP, consider uncertainty operators $\DDelta_A, \DDelta_B \in \mathcal{L}_{TV}^T$. Let ${\Delta_A}_t = \DDelta_A(t,:), {\Delta_B}_t = \DDelta_B(t,:)$ and ${\Delta_A}_t(x_{0:t}) = \sum_{i = 0}^t \dsuper{\Delta_A}{t}{t-i} x_i, {\Delta_B}_t(u_{0:t}) = \sum_{i = 0}^t \dsuper{\Delta_B}{t}{t-i} u_i$. At time $k$, denote the current state as $x(k)$. Solve the following robust optimal control problem:
\begin{equation} \label{eq:robustMPCprob}
	\begin{aligned}
 \min_{\mathbf{\pi} \in \Pi  } & \quad J(x(k), \pi)\\
 \text{s.t.} & 
	\left.
	\begin{array}{l}
	x_{t + 1} = \hat{A}_t x_t + {\Delta_A}_t (x_{0:t}) + \hat{B}_t u_t + {\Delta_B}_t(u_{0:t}) + w_t \\
	u_t = \pi_t(x_{0:t}, u_{0:t-1}) \\
	x_t \in \mathcal{X}, u_t \in \mathcal{U}, x_T \in \mathcal{X}_T
	\end{array}
	\right\} \\
& \quad \forall w_t \in \mathcal{W}, \forall t = 0, 1, \cdots, T - 1 \\
 & \quad  \forall \lVert {\DDelta_A} \rVert_\indinf \leq \epsilon_A, \forall \lVert {\DDelta_B} \rVert_\indinf \leq \epsilon_B   \\
 & \quad  x_0 = x(k)
	\end{aligned}
\end{equation}
where $\mathbf{\pi} = \pi_{0:T-1} \in \Pi$ for $\Pi$ the set of causal control policies, $J(x(k), \pi)$ is the cost function to be specified, $(\hat{A}_t, \hat{B}_t)$ is the known LTV nominal model, ${\DDelta_A}, {\DDelta_B}$ are the norm-bounded uncertainty operators and $\mathcal{X}, \mathcal{U}, \mathcal{X}_T$ are polytopic state, input and terminal constraints defined as
\begin{equation*} \label{eq:constraints}
\begin{aligned}
&\mathcal{X} = \lbrace x \in \mathbb{R}^{n} \mid F_x x \leq b_x \rbrace, \  
\mathcal{U} = \lbrace u \in \mathbb{R}^{m} \mid F_u u \leq b_u \rbrace,\\
&\mathcal{X}_T = \lbrace x \in \mathbb{R}^{n} \mid F_T x \leq b_T \rbrace.
\end{aligned}
\end{equation*}
Additionally, we define the disturbance set $\mathcal{W}$ as in~\eqref{eq:wSet}. 
\end{problem}
\begin{remark} \label{remark:memoryless}
The robust OCP formulation~\eqref{eq:robustMPCprob} contains the uncertainty model \eqref{eq:UncertaintySet} as a special case when the predictive model is of the form
\begin{equation*}
x_{t + 1} = (\hat{A}_t + \dsuper{\Delta_A}{t}{0}) x_t + (\hat{B}_t + \dsuper{\Delta_B}{t}{0})u_t + w_t
\end{equation*}
where we restrict $\DDelta_A, \DDelta_B$ to be block diagonal. 
\end{remark}

In Problem~\ref{prob:robustConstrControl}, the cost $J(x(k), \pi)$ is a function of the initial condition and the applied control policies. For simplicity, we apply the nominal cost $J_{\text{nom}}(x(k), \pi)$:
\begin{equation*} \label{eq:nominalcost}
\begin{array}{rl}
J_{\text{nom}}(x(k), \pi) & =  \sum_{t = 0}^{T - 1} (\bar{x}_t^\tp Q_t \bar{x}_t + u_t^\tp R_t u_t ) + \bar{x}_T^\tp Q_T \bar{x}_T \\
\text{s.t.} \ & \bar{x}_{t+1} = \hat{A}_t \bar{x}_t + \hat{B}_t u_t,  u_t = \pi_t(\bar{x}_{0:t}, u_{0:t-1}) \\
& \bar{x}_0 = x(k), \quad \forall t = 0, 1, \cdots, T-1
\end{array}
\end{equation*}
for $\bar{x}_t$ the nominal trajectory and $Q_t \succeq 0, R_t \succ 0, Q_T \succeq 0$  the state, input and terminal weight matrices, respectively.

We remark that in Problem~\ref{prob:robustConstrControl} we optimize over the closed-loop control policies $\pi\in \Pi$ instead of open-loop control sequences $u_{0:T-1}$, as open-loop policies are overly conservative~\cite{bemporad1999robust}. Since the space of all causal control policies is infinite-dimensional, we restrict our search to causal LTV state feedback controllers.\footnote{Affine feedback control policies can be implemented as linear feedback controllers by augmenting the state to $\tilde{x} = [x; 1]$. As such, we restrict our discussion to linear feedback policies without loss of generality.}

\section{System Level Synthesis}
\label{sec:SLSpreliminaries}

In this section, we review relevant concepts of SLS, and show how it can be used to reformulate OCPs over closed-loop system responses. A more extensive review of SLS can be found in~\cite{ANDERSON2019364}.

\subsection{Finite Horizon System Level Synthesis}
For the dynamics in~\eqref{eq:RealSys} and a fixed horizon $T$, we define $\xx$, and $\uu$ as the stacked states $x_{0:T}$, and inputs $u_{0:T}$, up to time $T$, i.e.,
\begin{equation*} 
\xx^\top = \begin{bmatrix} x_0^\top & x_1^\top & \cdots & x_T^\top \end{bmatrix}, \quad 
\uu^\top = \begin{bmatrix} u_0^\top & u_1^\top & \cdots & u_T^\top \end{bmatrix}.
\end{equation*}
We also let $\ww = [x_0^\tp \ w_0^\tp \cdots w_{T-1}^\tp]^\tp$, and note that we embed $x_0$ as the first component of the disturbance process. Let $u_t$ be a causal LTV state feedback controller, i.e., $u_t  = K_t(x_{0:t}) = \sum_{i = 0}^{t} \dsuper{K}{t}{t - i} x_i$. Then we have $\uu = \KK \xx$ with $\KK \in \mathcal{L}_{TV}^{T, m \times n}$.
We can concatenate the dynamics matrices as $\sA = \text{blkdiag}(A_0, \cdots, A_{T-1}, 0), \sB = \text{blkdiag}(B_0, \cdots, B_{T-1}, 0)$. 
Let $Z$ be the block-downshift operator, i.e., a matrix with the identity matrix on the first block subdiagonal and zeros elsewhere. Under the feedback controller $\KK$, the closed-loop behavior of the system~\eqref{eq:RealSys} over the horizon of length $T$ can be represented as 
\begin{equation}
	\xx = Z(\sA + \sB\KK)\xx + \ww
\end{equation}
and the closed-loop transfer functions describing $\ww \mapsto (\xx, \uu)$  are given by 
\begin{equation} \label{eq:clMapK}
\begin{bmatrix}
\xx \\
\uu
\end{bmatrix} 
=
\begin{bmatrix}
(I - Z(\sA + \sB \KK))^{-1} \\
\KK (I - Z(\sA + \sB \KK))^{-1}
\end{bmatrix} \ww.
\end{equation}
These maps are called \emph{system responses}, as they describe the closed-loop system behavior. Note that as $Z$ is the block-downshift operator, the matrix inverse in equation~\eqref{eq:clMapK} exists. 

Instead of optimizing over feedback controllers $\KK$, SLS allows for a direct optimization over system responses $\Phix, \Phiu$ defined as 
\begin{equation} \label{eq:clMapPhi}
	    \begin{bmatrix}  \xx \\  \uu \end{bmatrix}  =
	\begin{bmatrix}  \Phix \\ \Phiu \end{bmatrix} \ww
\end{equation} 
where $\Phix \in \mathcal{L}_{TV}^{T, n \times n}, \Phiu \in \mathcal{L}_{TV}^{T, m \times n}$ are two block-lower-triangular matrices by exploiting the following theorem. 
\begin{thm}~\cite[Theorem 2.1]{ANDERSON2019364} \label{thm:equivalence}
Over the horizon $ t = 0, 1, \cdots, T$, for the system dynamics~\eqref{eq:RealSys} with the block-lower-triangular state feedback control law $\KK \in \mathcal{L}_{TV}$ defining the control action as $\uu = \KK \xx$, the following are true:
\begin{enumerate}
	\item The affine subspace defined by 
	\begin{equation} \label{eq:PhiAffineSpace}
	\begin{bmatrix} I - Z \sA & -Z \sB \end{bmatrix} 
	\begin{bmatrix} \Phix \\ \Phiu \end{bmatrix} = I, \ \Phi_x, \Phi_u \in \mathcal{L}_{TV}
	\end{equation}
	parameterizes all possible system responses~\eqref{eq:clMapPhi}.
	\item For any block-lower-triangular matrices $\lbrace \Phix, \Phiu \rbrace \in \mathcal{L}_{TV}$ satisfying~\eqref{eq:PhiAffineSpace}, the controller $\KK = \Phiu \Phix^{-1} \in \mathcal{L}_{TV}$ achieves the desired response. 
\end{enumerate}
\end{thm}
Theorem~\ref{thm:equivalence} states that the problem of controller synthesis can be equivalently posed as one over closed-loop system responses, by setting $\xx = \Phix \ww$, $\uu = \Phiu \ww$, and constraining that the system responses lie in the affine space \eqref{eq:PhiAffineSpace}. The map~\eqref{eq:clMapPhi} characterizes the effects of disturbances $\ww$ on the states $\xx$ and inputs $\uu$ and thus allows a direct translation of the state and input constraints into constraints on $\Phix, \Phiu$. 

\begin{remark}
Theorem~\ref{thm:equivalence} reveals the equivalence between optimizing over the state feedback controller $\uu = \KK\xx$ and the disturbance feedback controller $\uu = \Phi_u \ww$, when there is no model uncertainty in problem~\eqref{eq:robustMPCprob}. Such equivalence has also been shown in~\cite{goulart2006optimization} as one of the main results.
\end{remark}


\subsection{Robustness in System Level Synthesis}
A useful result in characterizing the robustness of the closed-loop system responses is presented here. When the  constraint~\eqref{eq:PhiAffineSpace} is not exactly satisfied, the system responses can be described in the following theorem:
\begin{thm}~\cite[Theorem 2.2]{ANDERSON2019364} \label{thm:robust}
Let $\bar{\DDelta} \in \mathcal{L}_{TV}$ be an arbitrary block-lower-triangular matrix as in equation~\eqref{eq:BLT}. Suppose $\{\Phix, \Phiu\} \in \mathcal{L}_{TV}$ satisfy
\begin{equation} \label{eq:DDeltaEq}
\begin{bmatrix} I - Z \sA & -Z \sB \end{bmatrix} 
\begin{bmatrix} \Phix \\ \Phiu \end{bmatrix} = I - \bar{\DDelta}.
\end{equation}
If $(I - \dsuper{\bar{\Delta}}{i}{0})^{-1}$ exists for $i = 0, \cdots, T$, then the controller $\KK = \Phiu \Phix^{-1} \in \mathcal{L}_{TV}$ achieves the system response
\begin{equation} \label{eq:robustSysResp}
    \begin{bmatrix}  \xx \\  \uu \end{bmatrix}  =
	\begin{bmatrix}  \Phix \\ \Phiu \end{bmatrix} ( I - \bar{\DDelta})^{-1} \ww.
\end{equation}
\end{thm}
Theorem~\ref{thm:robust} describes how the perturbations in the constraint affect the closed-loop responses. Next, we use $\bar{\DDelta}$ to capture the uncertainty in the system dynamics, and the expression~\eqref{eq:robustSysResp} affords us a transparent mapping of its effects on system behavior.


\section{Robust MPC with Model Uncertainty}
\label{sec:SLSuncertainty}

In this section, we show how SLS can be applied to the robust MPC problem in the general setting of additive disturbance and model uncertainty, and show that a suitable relaxation of the resulting robust SLS OCP results in a quasi-convex optimization problem.

\subsection{Robustness in Nominal Control Synthesis}
For the nominal model $\hat{\sA} = \text{blkdiag}(\hat{A}_0, \cdots, \hat{A}_{T-1}, 0)$ and $\hat{\sB} = \text{blkdiag}(\hat{B}_0, \cdots, \hat{B}_{T-1}, 0)$, consider block-lower-triangular matrices $ \{ \hat{\mathbf{\Phi}}_x, \hat{\mathbf{\Phi}}_u \}$ satisfying:
\begin{equation} \label{eq:hatPhixConstr}
\begin{bmatrix} I - Z \hat{\sA} & -Z \hat{\sB} \end{bmatrix}
\begin{bmatrix} \hat{\mathbf{\Phi}}_x \\ \hat{\mathbf{\Phi}}_u \end{bmatrix} = I, \ \hat{\mathbf{\Phi}}_x, \hat{\mathbf{\Phi}}_u \in \mathcal{L}_{TV}.
\end{equation}
The nominal responses $\{\hat{\mathbf{\Phi}}_x, \hat{\mathbf{\Phi}}_u \}$ approximately satisfy the constraint~\eqref{eq:PhiAffineSpace} with respect to the real dynamics $\sA = \hat{\sA} + \DDelta_A,\,  \sB =  \hat{\sB} + \DDelta_B$, where the uncertainty operators $\{ \DDelta_A, \DDelta_B \}$ are defined in Problem~\ref{prob:robustConstrControl}. This can be seen by rewriting~\eqref{eq:hatPhixConstr} as 
\begin{equation} \label{eq:inexact}
\begin{aligned}
\begin{bmatrix} I - Z \sA & -Z \sB \end{bmatrix}
\begin{bmatrix} \hat{\mathbf{\Phi}}_x \\ \hat{\mathbf{\Phi}}_u \end{bmatrix} & = I - Z \begin{bmatrix} \DDelta_A & \DDelta_B \end{bmatrix} \begin{bmatrix} \hat{\mathbf{\Phi}}_x \\ \hat{\mathbf{\Phi}}_u \end{bmatrix} \\
& = I - \DeltaVar \hat{\PPhi} 
\end{aligned}
\end{equation}
with the notation
\begin{equation} \label{eq:defPhiDelta}
	\hat{\mathbf{\Phi}} = \begin{bmatrix}
	\hat{\mathbf{\Phi}}_x \\ \hat{\mathbf{\Phi}}_u
	\end{bmatrix}, \quad \DeltaVar = Z \begin{bmatrix}
	\DDelta_A & \DDelta_B
	\end{bmatrix}.
\end{equation}

Thus, $\{ \hat{\mathbf{\Phi}}_x, \hat{\mathbf{\Phi}}_u \}$ satisfy an inexact constraint~\eqref{eq:DDeltaEq} for the system $(\sA, \sB)$ with $\bar{\DDelta} := \DeltaVar \hat{\mathbf{\Phi}}$. Since $\bar{\DDelta}$ is a strict block-lower-triangular matrix and $(I - \dsuper{\bar{\Delta}}{i}{0})^{-1}$ exists for all $i$, by Theorem~\ref{thm:robust} the closed-loop system response of the system $(\sA, \sB)$ with the controller $\hat{\KK} = \hat{\mathbf{\Phi}}_u \hat{\mathbf{\Phi}}_x^{-1}$ is given by

\begin{equation} \label{eq:clMapUncertainty}
\begin{aligned}
\begin{bmatrix}  \xx \\  \uu \end{bmatrix} & = \hat{\mathbf{\Phi}} ( I - \DeltaVar \hat{\mathbf{\Phi}})^{-1} \ww \\
& = (\hat{\mathbf{\Phi}} + \hat{\mathbf{\Phi}} \DeltaVar( I - \hat{\mathbf{\Phi}} \DeltaVar)^{-1} \hat{\mathbf{\Phi}}) \ww
\end{aligned}
\end{equation}
where the second equality is an application of the Woodbury matrix identity~\cite{horn2012matrix}.
Noting that the second line of equation~\eqref{eq:clMapUncertainty} is a linear fractional transform (LFT) of $\DDelta$ with an approximately constructed augmented plant defined in terms of the system response $\hat{\PPhi}$, we leverage tools from robust control~\cite{Dahleh1994ControlOU} to account for model uncertainty while introducing a minimal amount of conservativeness in a computationally efficient way. 

\subsection{Upper Bound through Robust Performance}
\label{sec:RobustPerformance}
In this section, we provide an easily-verified sufficient condition for the $\ell_{\indinf} \rightarrow \ell_{\indinf}$ norm of the system response~\eqref{eq:clMapUncertainty} to be upper bounded by a constant. This condition exploits the LFT structure of~\eqref{eq:clMapUncertainty} and is inspired by the algebraic robust performance conditions in~\cite{khammash1991performance}.

Define the perturbation set $\DDelta_c = \{ \DDelta \in \mathcal{L}_{TV} \mid \lVert \DDelta \rVert_\indinf \leq 1, \DDelta \text{ is strictly causal}  \}$. 
\begin{proposition}~\cite{dullerud2013course} \label{prop:equivalence}
Consider System I and System II in Fig.~\ref{fig:sysIandII} where
\begin{equation} \label{eq:Mdef}
	\MM = \begin{bmatrix} \MM_{11} & \MM_{12} \\ \MM_{21} & \MM_{22} \end{bmatrix}, \quad
	\widetilde{\DDelta} = \begin{bmatrix}
	\DDelta_1 & \\ & \DDelta_2
	\end{bmatrix}
\end{equation}
with $\DDelta_1, \DDelta_2 \in \DDelta_c$, $\MM_{ij} \in \mathcal{L}_{TV}, i, j \in \{1, 2\} $. The matrices are of compatible dimensions. Then the following statements are equivalent:
\begin{enumerate}
	\item In System I, $(I - \MM_{22} \DDelta_1)^{-1}$ exists and 
	\begin{equation}
		\lVert \MM_{11} + \MM_{12} \DDelta_1 (I - \MM_{22}\DDelta_1)^{-1}\MM_{21} \rVert_\indinf < 1
	\end{equation}
	for all $\DDelta_1 \in \DDelta_c$.
	\item In System II, $(I - \MM \widetilde{\DDelta})^{-1}$ exists for all $\widetilde{\DDelta} = \text{blkdiag}(\DDelta_1, \DDelta_2), \DDelta_1 , \DDelta_2 \in \DDelta_c$.
\end{enumerate}
\end{proposition}

\begin{figure}
	\begin{subfigure}[b]{0.49 \columnwidth}\hfil
\begin{tikzpicture}[auto, node distance=2.2cm,>=latex']
\node [input, name=u]{$u$};
\node [block, right of=u] (G) {$\MM$};
\node [output, right of=G, name=x] {$x$};
\node [block1, below of=G, node distance=1.5cm] (phi) {$\DDelta_1$};
\node [input, right of=phi,name=xx]{$x$};

\draw (0.85,0.4) node[]{$\ww$};  
\draw[->] (0.65,0.25) to (1.85,0.25){};
\draw (3.35,0.4) node[]{$\mathbf{z}$};  
\draw[->] (2.55,0.25) to (3.6,0.25){};
\draw[] (2.55,-0.15) to (3.5,-0.15){};

\draw[] (3.5,-0.15) to (3.5,-1.5){};
\draw[->] (3.5,-1.5) to (2.55,-1.5){};
\draw[] (1.85,-1.5) to (0.75,-1.5){};

\draw[] (0.75,-1.5) to (0.75,-0.15){};

\draw[->] (0.75,-0.15) to (1.85,-0.15){};



\end{tikzpicture}
\caption{System I}
	\end{subfigure} 
\begin{subfigure}[b]{0.49 \columnwidth} \hfil
	\begin{tikzpicture}[auto, node distance=2.2cm,>=latex']
	\node [input, name=u]{$u$};
	\node [block, right of=u] (G) {$\MM$};
	\node [output, right of=G, name=x] {$x$};
	\node [block1, below of=G, node distance=1.5cm] (phi) {$\widetilde{\DDelta}$};
	\node [input, right of=phi,name=xx]{$x$};


	\draw[] (2.55,-0.05) to (3.5,-0.05){};
	
	\draw[] (3.5,-0.05) to (3.5,-1.5){};
	\draw[->] (3.5,-1.5) to (2.55,-1.5){};
	\draw[] (1.85,-1.5) to (0.75,-1.5){};
	
	\draw[] (0.75,-1.5) to (0.75,-0.05){};
	
	\draw[->] (0.75,-0.05) to (1.85,-0.05){};
	\end{tikzpicture}
	\caption{System II}
\end{subfigure} 
\caption{Left: System I. Right: System II.}
\label{fig:sysIandII}
\end{figure}
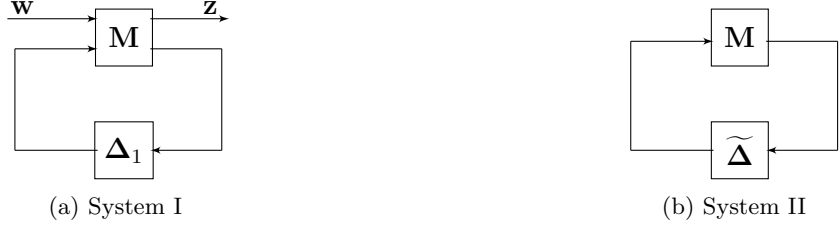

We state our sufficient condition in the following theorem:
\begin{thm}\label{thm:robustPerformance}
Consider the matrices $\hat{\PPhi}$ and $\DeltaVar$ in~\eqref{eq:clMapUncertainty}. Let $\lVert \DeltaVar \rVert_\indinf \leq \epsilon$ for some $\epsilon > 0$. For any $\beta > 0 $ and matrix $\mathbf{P} \in \mathcal{L}_{TV}$ of compatible dimensions, we have 
\begin{equation} \label{eq:robustRespConstr}
	\lVert \mathbf{P} \hat{\PPhi} + \mathbf{P} \hat{\PPhi} \DeltaVar (I - \hat{\PPhi} \DeltaVar)^{-1} \hat{\PPhi} \rVert_\indinf \leq \beta, \forall \lVert \DeltaVar \rVert_\indinf \leq \epsilon
\end{equation}
\begin{equation} \label{eq:robustPerfConstr} 
\begin{aligned}
&\text{if} \quad \quad \quad \quad \quad \lVert \mathbf{P} \hat{\PPhi} \rVert_\indinf + \beta \epsilon \lVert \hat{\PPhi} \rVert_\indinf \leq \beta. 
\end{aligned}
\end{equation}
\end{thm}
\begin{proof}
	See Appendix.
\end{proof}
Theorem~\ref{thm:robustPerformance} provides a sufficient condition~\eqref{eq:robustPerfConstr} which is \emph{convex} in $\hat{\PPhi}$ for the robust norm bound constraint on the system responses~\eqref{eq:robustRespConstr} to be satisfied. 

\subsection{Robust Optimization Formulation} 
\label{sec:RobustOpt}
We leverage Theorems~\ref{thm:equivalence} $\sim$ \ref{thm:robustPerformance}, and equation~\eqref{eq:clMapUncertainty} to reformulate Problem~\ref{prob:robustConstrControl} as a robust optimization problem over the system responses. The remaining challenge is to separate out the effects of the \emph{known} initial condition $x_0$ from the \emph{unknown but adversarial} future disturbance $w_{0:T-1}$. First, we concatenate all the constraints $x_t \in \mathcal{X}, u_t \in \mathcal{U}, x_T \in \mathcal{X}_T$ in Problem~\eqref{eq:robustMPCprob} on $x_t$ and $u_t$ over the horizon $T$ as
\begin{equation} \label{eq:stackconstr}
	F \begin{bmatrix}
	\xx \\ \uu
	\end{bmatrix} \leq b,
\end{equation} 
where
\begin{equation}
\begin{aligned}
& F = \text{blkdiag}(F_x, \cdots, F_x, F_T, F_u, \cdots, F_u), \\
& b = [b_x^\tp \ \cdots \ b_x^\tp \ b_T^\tp \ b_u^\tp \ \cdots \ b_u^\tp]^\tp.
\end{aligned}
\end{equation}
We decompose $\PPhi, \DeltaVar$, and $\ww$ as follows to separate the effects of the known initial condition $x_0$ from the unknown future disturbances $w_{0:T-1}$:
\begin{equation}\label{eq:matStructure}
\begin{aligned}
& \PPhi = \begin{bmatrix} \Phix \\ \Phiu \end{bmatrix} = \left[ \begin{array}{c | c} \PPhi^0 & \PPhi^\tildeww \end{array} \right] = 
\left[ \begin{array}{c | c} \Phix^0 & \Phix^\tildeww \\ \hline \Phiu^0 & \Phiu^\tildeww \end{array} \right], \ww = \begin{bmatrix} x_0 \\ \tildeww \end{bmatrix}, \\
& \DeltaVar = \left[  \begin{array}{c}
\DeltaVar^0  \\ \hline
\DeltaVar^\tildeww
\end{array} \right] 
= \left[ \begin{array}{c | c}
\DeltaVar_A & \DeltaVar_B
\end{array}     \right] 
= \left[ \begin{array}{c | c}
\DeltaVar_A^0 & \DeltaVar_B^0 \\ \hline \DeltaVar_A^\tildeww & \DeltaVar_B^\tildeww
\end{array}   \right],
\end{aligned}
\end{equation}
where $\PPhi^0$ is the first block column of $\PPhi$, $\DeltaVar^0$ is the first block row of $\DeltaVar$ and all other block matrices are of compatible dimensions. The main result of this paper is presented below:
\begin{thm} \label{thm:robustConvexRelaxation}
	Consider the convex optimization problem formulated below:
\begin{subequations}\label{eq:robustRelaxation}
	\begin{align} 
	\begin{split}
	\underset{\PPhi}{\min} & \quad \Bigg\lVert 
	\begin{bmatrix} \sQ^{\frac{1}{2}} & \ \\ \ & \sR^{\frac{1}{2}} \end{bmatrix} 
	\begin{bmatrix} \Phix(:,0) \\ \Phiu(:,0) \end{bmatrix} x_0
	\Bigg\rVert_2^2 \label{constr:obj} 	\end{split} \\
	\begin{split}
	\text{s.t.} & \quad \begin{bmatrix} I - Z \hat{\sA} & -Z \hat{\sB} \end{bmatrix}
	\begin{bmatrix} \Phix \\ \Phiu \end{bmatrix} = I \label{constr:affine} \end{split}\\
	\begin{split}
	& \quad  F_j^\tp \PPhi^0 x_0 +\lVert F_j^\tp \PPhi^\tildeww\rVert_1 \frac{1 - \tau^T}{1 - \tau} \gamma + \beta \sigma_w \leq b_j, \forall j \label{constr:safety} \end{split} \\
	\begin{split}
	& \quad \lVert F_j^\tp \PPhi^\tildeww \rVert_\indinf + \beta\lVert \epsilon \PPhi^\tildeww \rVert_\indinf \leq \beta, \forall j  \label{constr:beta} \end{split} \\
	\begin{split}
	& \quad \Big\lVert \begin{bmatrix} \epsilon_A \Phix^\tildeww \\  \epsilon_B \Phiu^\tildeww \end{bmatrix} \Big\rVert_\indinf \leq \tau/2 \label{constr:tau} \end{split}\\
	\begin{split}
	& \quad \Big\lVert \begin{bmatrix} \epsilon_A \Phix^0 \\  \epsilon_B \Phiu^0 \end{bmatrix} x_0 \Big\rVert_\infty \leq \gamma/2 \label{constr:gamma} \end{split}
	\end{align}
\end{subequations}
	with constants $\epsilon = \epsA + \epsB, \tau > 0, \gamma > 0$ and $\beta > 0$. Any controller synthesized from a feasible solution $\PPhi$ to~\eqref{eq:robustRelaxation} for any $(\tau, \gamma, \beta)$ guarantees the constraints satisfaction under all possible model uncertainty and disturbances in~\eqref{eq:robustMPCprob}.
\end{thm}
\begin{proof}
See Appendix.
\end{proof} 

Despite the existence of three hyperparameters $(\tau, \gamma, \beta)$, in practice \eqref{eq:robustRelaxation} can be effectively solved by grid searching over $(\tau, \gamma, \beta)$. To reduce the range over which the grid search must be conducted, upper and lower bounds on $(\tau, \gamma, \beta)$ can be efficiently found via bisection, which has $O(\log(1/\epsilon_{tol}))$ complexity in terms of the tolerance $\epsilon_{tol}$. Algorithm~\ref{alg:bisection} describes this procedure using the standard bisection algorithm \textbf{bisect}. When applying \textbf{bisect}, a variant of~\eqref{eq:robustRelaxation}, which is a linear program, is constructed with a smaller set of selected constraints whose feasibility guides the search of the lower or upper bounds on the hyperparameters. Since~\eqref{constr:tau},~\eqref{constr:beta} are independent of $x_0$, the lower bounds on $\tau, \beta$ obtained from Algorithm~\ref{alg:bisection} are valid for all $x_0$ and can be computed offline. 
\begin{algorithm}
	\caption{Find lower and upper bounds on $(\tau, \gamma, \beta)$}\label{alg:bisection}
	\hspace*{\algorithmicindent} \textbf{Input:} $\epsilon_{tol} > 0,$ range($\tau, \gamma, \beta$) \\
	\hspace*{\algorithmicindent} \textbf{Output:} $lb(\tau, \gamma, \beta)$, $ub(\tau, \gamma, \beta)$ 
	\begin{algorithmic}[1]
		\Procedure{LowerUpperBounds}{$\tau, \gamma, \beta$}~\footnotemark
		\State $lb(\tau) = $\textbf{ bisect}($\tau, \{ \min_\PPhi  0 \ \text{s.t. } \eqref{constr:affine}, \eqref{constr:tau} \}$)
		\State $lb(\gamma) = $\textbf{ bisect}($\gamma, \{ \min_\PPhi  0 \ \text{s.t. } \eqref{constr:affine}, \eqref{constr:gamma} \}$)
		\State $lb(\beta) = $\textbf{ bisect}($\beta, \{ \min_\PPhi  0 \ \text{s.t. } \eqref{constr:affine}, \eqref{constr:beta} \}$)
		\State In~\eqref{constr:safety}, let $\tau = lb(\tau), \gamma = lb(\gamma)$
		\State $ub(\beta) = $\textbf{ bisect}($\beta, \{ \min_\PPhi  0 \ \text{s.t. } \eqref{constr:affine}, \eqref{constr:safety} \}$)
		\State In~\eqref{constr:safety}, let $\gamma = lb(\gamma), \beta = lb(\beta)$
		\State $ub(\tau) = $\textbf{ bisect}($\tau, \{ \min_\PPhi 0 \ \text{s.t. } \eqref{constr:affine}, \eqref{constr:safety} \}$)
		\State $ub(\gamma) = \frac{1 - ub(\tau)^T}{1 - ub(\tau)} lb(\gamma)/\frac{1 - lb(\tau)^T}{1 - lb(\tau)} $
		\State \textbf{return} $lb(\tau, \gamma, \beta), ub(\tau, \gamma, \beta)$ 
		\EndProcedure
	\end{algorithmic}
\end{algorithm}
\footnotetext{All the bisection algorithms search over the range($\tau, \gamma, \beta$) and use $\epsilon_{tol}$ for termination as specified in the input.}

Furthermore, with Algorithm~\ref{alg:bisection} we can verify the infeasibility of the relaxation~\eqref{eq:robustRelaxation} for all choices of hyperparameters if the contradiction $lb(\cdot) > ub(\cdot)$ is observed for $\tau, \gamma$ or $\beta$. With fixed hyperparameters, problem~\eqref{eq:robustRelaxation} is a quadratic program which has the time complexity $O(T^6(n(n+m))^3)$.

Compared with the previous SLS-based robust optimization relaxations~\cite{dean2019safely} and tube-based robust MPC~\cite{langson2004robust}, problem~\eqref{eq:robustRelaxation} significantly reduces the conservativeness by separately bounding the effects of the disturbance $\tilde{\ww}$, which can then be handled near optimally leveraging tools from $\mathcal{L}_1$ robust control~\cite{Dahleh1994ControlOU}, the effects of the initial condition $x_0$, and by introducing hyperparameters $\beta$ and $\gamma$ to accommodate a wide range of $\lVert x_0 \rVert_\infty$ sizes. In Section~\ref{sec:methodscomparison}, we provide a numerical example demonstrating these claims, and further showing that despite the introduction of grid searches over three hyperparameters, our method is still computationally more efficient and less conservative than tube MPC.


\section{Numerical Examples}
\label{sec:examples}

We consider robust MPC of a double integrator system \footnote{We note that the double integrator is a standard example system used to validate robust MPC methods subject to both additive disturbance and model uncertainty~\cite{mayne2005robust, schuurmans2000robust, kouvaritakis2000efficient}. That all papers in the literature are limited to such small scale examples~\cite{mayne2005robust, schuurmans2000robust,kouvaritakis2000efficient, kothare1996robust, lee2000robust, bujarbaruah2019semi, lorenzen2019robust} highlights the difficulty of the problem, and why it is still an active area of research.} as an illustrating example where both the SLS approach and the tube MPC approach~\cite{langson2004robust} are implemented in Section~\ref{sec:numericalsetup}. In Section~\ref{sec:methodscomparison}, a detailed discussion of the SLS, tube MPC, and dynamic programming approaches is conducted. Furthermore, we will demonstrate that our proposed relaxation~\eqref{eq:robustRelaxation} significantly reduces the conservativeness compared with the existing SLS relaxations in literature. All the simulation is implemented in MATLAB R$2018$a with YALMIP~\cite{lofberg2004yalmip} and MOSEK~\cite{andersen2000mosek} on an Intel i$7$-$6700$K CPU. All codes are publicly available at \url{https://github.com/unstable-zeros/robust-mpc-sls}.

\subsection{Robust MPC problem setup}
\label{sec:numericalsetup}
Consider the robust MPC problem~\eqref{eq:robustMPCprob} where the nominal system is given by
\begin{equation}\label{eq:example_1_dynamics}
\begin{aligned}
x_{t+1} & =  \hat{A} x_t + \hat{B} u_t + w_t = \begin{bmatrix} 1 & 1 \\ 0 & 1  \end{bmatrix} x_t + \begin{bmatrix} 0.5 \\ 1 \end{bmatrix} u_t + w_t.
\end{aligned}
\end{equation}
Although our approach can accommodate general LTV dynamics and uncertainty, we investigate an LTI double integrator system as our nominal system in this section to facilitate comparison to other methods in the literature. Since the tube MPC approach can only handle memoryless uncertainty, we let the uncertainty operator $\DDelta_A, \DDelta_B$ in~\eqref{eq:robustMPCprob} be memoryless, as in Remark~\ref{remark:memoryless}. The uncertainty level is set as $\epsA = 0.02, \epsB = 0.05$, corresponding to up to $5\%$ parametric uncertainty in $\hat{B}$ and up to $2\%$ parametric uncertainty in $\hat{A}$. The nominal cost function $J_{nom}$ is defined by $Q = I_2$, $R = 0.1$, and $Q_T = Q$. 

The state and input constraints are given as:
\begin{equation} 
\begin{aligned}
& \mathcal{X} = \Big\lbrace x \in \mathbb{R}^2 \mid \begin{bmatrix} -10 \\ -10 \end{bmatrix} \leq x \leq 
\begin{bmatrix} 10 \\ 10 \end{bmatrix}
\Big\rbrace, \\
& \mathcal{U} = \Big\lbrace u \in \mathbb{R} \mid -2 \leq x \leq 2
\Big\rbrace. \\
\end{aligned}
\end{equation}
The additive disturbance is norm-bounded by $w_t \in \mathcal{W} = \Big\lbrace w \in \mathbb{R}^2 \mid \lVert w \rVert_\infty \leq 0.1
\Big\rbrace$ and the terminal set $\mathcal{X}_T$ is found to be the maximal robust forward invariant set of system~\eqref{eq:example_1_dynamics} through dynamic programming~\cite{borrelli2017predictive}. We plot $\mathcal{X}_T$ as the green region and $\mathcal{X}$ as the purple region in Fig.~\ref{fig:Zunit}. By definition of $\mathcal{X}_T$, any $x_0$ outside $\mathcal{X}_T$ can not be robustly steered into $\mathcal{X}_T$ under any admissible control actions. 

%
The \emph{tube MPC} approach proposed in~\cite{langson2004robust} aims to bound the spread of trajectories under uncertainty by a sequence of polytopic tubes of different sizes. The tube shape $Z = \text{conv}\{z_1, \cdots, z_V\}$, which is the convex hull of $V$ vertices, has to be fixed beforehand. In our example, we test two kinds of tubes: one is the unit ball with $\infty$-norm $Z_{\text{unit}} = \{\lVert x \rVert_\infty \leq 1 \}$, and the other is the disturbance invariant set $Z_{\text{inv}} = \sum_{i = 1}^\infty (\hat{A} + \hat{B}K)^i \mathcal{W}$ for the closed-loop system $x_{t+1} = (\hat{A} + \hat{B}K)x_t + w_t$, where $K$ is LQR controller of the system $(\hat{A}, \hat{B})$ with the weight matrices $Q$ and $R$. The sum in the definition of $Z_{\text{inv}}$ is the Minkowski sum of sets~\cite{borrelli2017predictive}. For the details of implementing tube MPC, please see~\cite{langson2004robust}.

We highlight that the performance of tube MPC depends on the prefixed shape of the tube. In Fig.~\ref{fig:tubetraj}, we fix $T = 6$, $x_0 = [-7 \ -2]^\tp$, and simulate $10$ trajectories with randomly generated uncertainty parameters of tube MPC applied on system~\eqref{eq:example_1_dynamics} with different tubes $Z_{\text{unit}}$ and $Z_{\text{inv}}$. It is observed that with $Z_{\text{unit}}$, the trajectories prediction diverges and tube MPC becomes infeasible for horizon $T > 7$, while with $Z_{\text{inv}}$, tube MPC remains feasible for large horizons.

\begin{figure} [hbt!]
	\centering
	\begin{subfigure}{0.49 \columnwidth}\hfil
		\includegraphics[width= \linewidth]{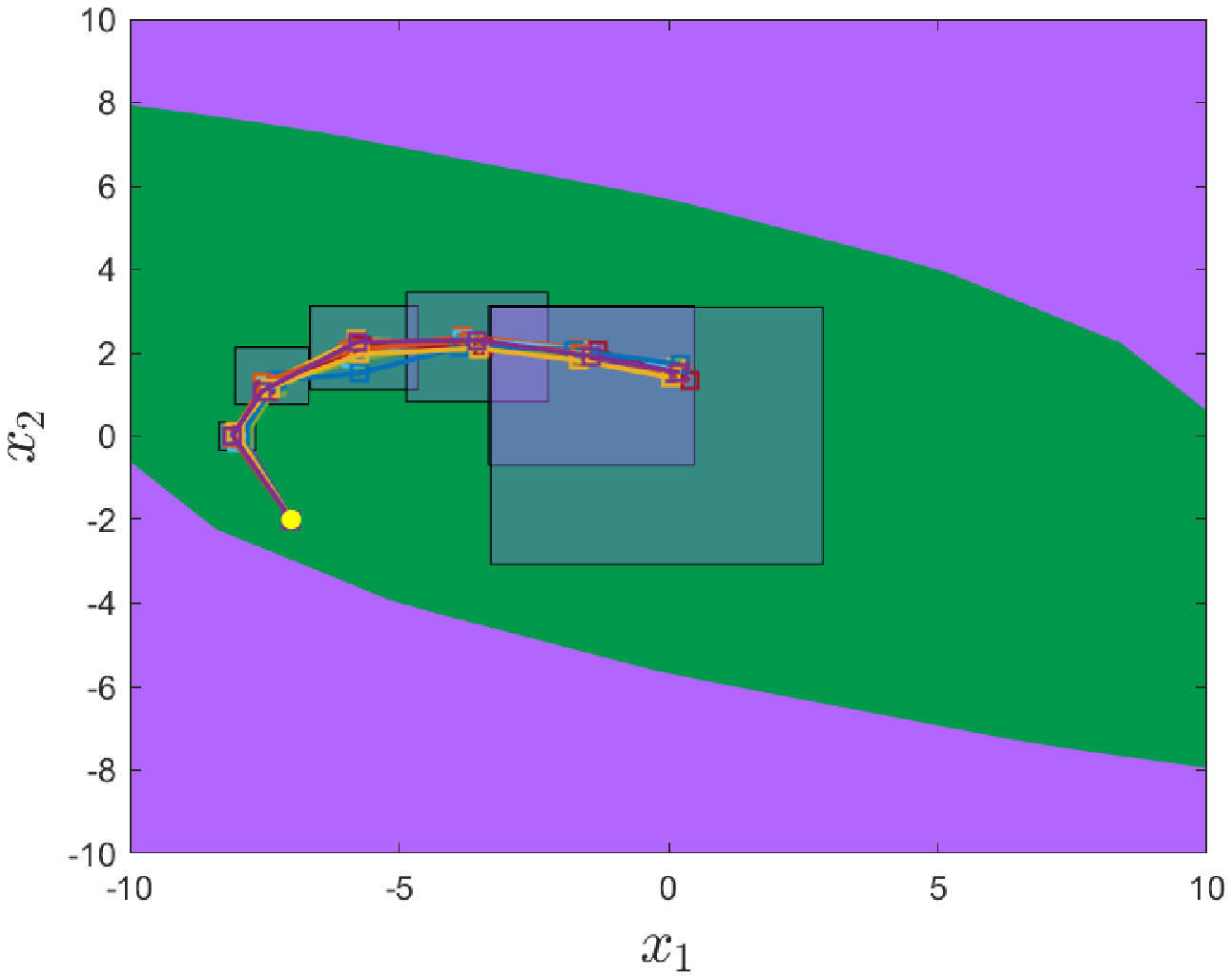}
		\caption{Tube MPC with $Z_{\text{unit}}$.}
		\label{fig:Zunit}
	\end{subfigure}
	\begin{subfigure}{0.49 \columnwidth}
		\includegraphics[width= \linewidth]{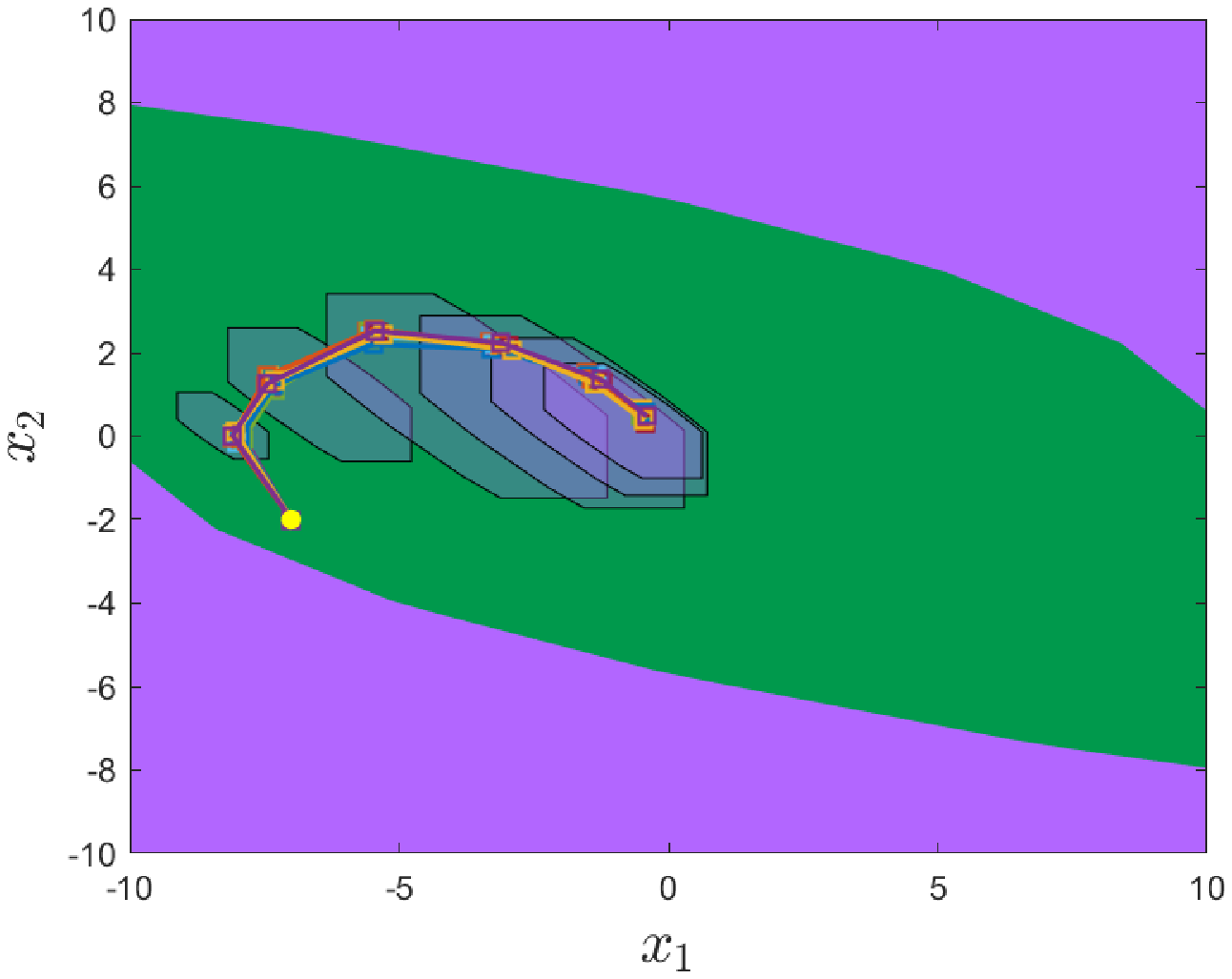}
		\caption{Tube MPC with $Z_{\text{inv}}$}
		\label{fig:Zinv}
	\end{subfigure}
	\caption{ Simulated trajectories of tube MPC with $Z_{\text{unit}}$ (left) and $Z_{\text{inv}}$ (right). The initial condition is marked by the yellow dot. The terminal constraint $\mathcal{X}_T$ is represented by the green region while the state constraint $\mathcal{X}$ is marked by the purple region.}
	\label{fig:tubetraj}
\end{figure}


\subsection{Methods comparison}
\label{sec:methodscomparison}
In this section, we compare the robust SLS and tube MPC approaches to solving robust OCPs as in Problem~\ref{prob:robustConstrControl}, with a focus on handling norm-bounded uncertainty $\DDelta_A$ and $\DDelta_B$.  We use the dynamic programming (DP) solution, which is well known to be computationally intractable~\cite{wang2009fast, borrelli2017predictive}, as a ground truth for comparison, and end the section with a brief discussion comparing our results to previous SLS based relaxatiosn of the robust OCP Problem~\ref{prob:robustConstrControl}.

\textbf{Dynamics and uncertainty assumptions:} All methods can handle LTV nominal dynamics.  DP and tube MPC methods can accommodate model uncertainty $\DDelta_A, \DDelta_B$ in the convex hull of a finite number of matrices. This generality comes at the cost of requiring vertex enumeration of the uncertainty polytope when formulating their optimization problems, which can lead to an exponential increase in the number of constraints as the system dimension increases, such as when the uncertainty is simply norm-bounded as in~\eqref{eq:robustMPCprob}. The SLS approach does not use vertex enumeration and the number of constraints grows approximately as $T(T+1)n^2/2$ - however, as of yet, it can only accommodate $\ell_\infty \rightarrow \ell_\infty$ norm bounded uncertainty model.

Although both DP and tube MPC allow model uncertainty to be time-variant, they can only handle memoryless uncertainty, i.e., $\DDelta_A, \DDelta_B$ must be block diagonal. SLS, on the other hand, can be applied even when $\DDelta_A, \DDelta_B$ are general LTV operators. 

\textbf{Complexity analysis:} 
Both robust SLS MPC~\eqref{eq:robustRelaxation} and tube MPC have a quadratic programming formulation with the number of variables listed in Table~\ref{table:compare}. The number of decision variables in the SLS approach is quadratic in the horizon $T$ and the state dimension $n$, while in tube MPC, it is linear in $T$ and $n$. However, when considering $\ell_\infty \rightarrow \ell_\infty$ norm-bounded uncertainty, the number of constraints grows exponentially with the state and input dimensions due to vertex enumeration. Furthermore, the number of vertices $V$ of the tube $Z$ and the number of the inequality constraints $H$ representing $Z$ act as a multiplier to the number of constraints as well.

In Fig.~\ref{fig:runningtime}, we plot the solver time of tube MPC with $Z_{\text{inv}}$ (blue) and robust SLS MPC (red) in solving the robust OCP of system~\eqref{eq:example_1_dynamics} with $x_0 = [-7 \ -2]^\tp$ and varying horizons. In tube MPC, we have $V = 20$ and $H = 20$ for $Z_{\text{inv}}$. In robust SLS MPC, we let $\epsilon_{tol} = 0.01$ and searching ranges of $(\tau, \gamma, \beta)$ be $[0, 10] \times [0, 10] \times [0, 10]$ in Algorithm~\ref{alg:bisection}. Then a $3 \times 3 \times 3$ grid search for feasible $(\tau, \gamma, \beta)$ is applied, followed by solving the quadratic program~\eqref{eq:robustRelaxation} with the found feasible hyperparameters. Note that for robust SLS MPC, the plotted value is the sum of all the solver times in the execution of Algorithm~\ref{alg:bisection}, the grid search stage, and the solving of~\eqref{eq:robustRelaxation}. Fig.~\ref{fig:runningtime} shows that even with the additional overhead incurred by the hyperparameter grid search, robust SLS MPC is computationally more efficient than tube MPC over a wide range of horizons in this robust MPC example.

\begin{figure}
	\centering
	\includegraphics[width = 0.6\columnwidth]{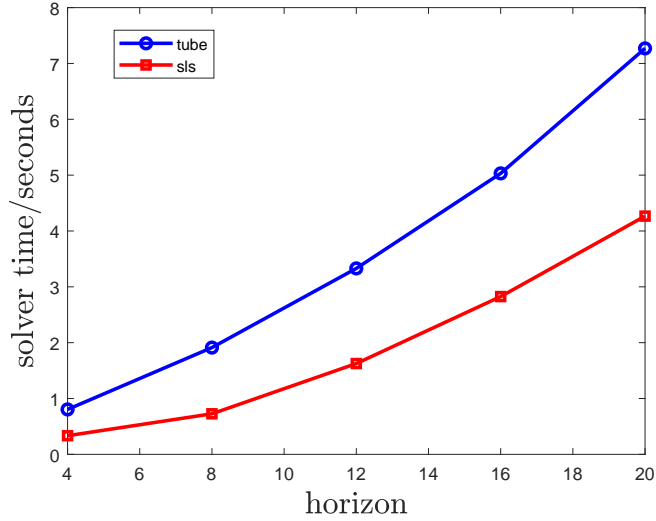}
	\caption{The solver time of running tube MPC (blue) and robust SLS MPC~\eqref{eq:robustRelaxation} (red) with different horizons. }
	\label{fig:runningtime}
\end{figure} 

\textbf{Conservativeness evaluation:} We evaluate the conservativeness of the SLS and tube MPC methods in terms of the feasibility of the initial conditions in the maximum robust forward invariant set $\mathcal{X}_T$ found through DP. We fix a horizon of $T = 10$ and use a grid over $x_0 \in \mathcal{X}_T$, as shown in Fig.~\ref{fig:slsfeasibility}. Using the SLS approach (Fig.~\ref{fig:slsfeasibility}), for each $x_0$, Algorithm~\ref{alg:bisection} is applied first with $\epsilon_{tol} = 0.01$ and searching ranges as $[0, 10] \times [0, 10] \times [0, 10]$, followed by a $ 3 \times 3 \times 3$ grid search of $(\tau, \gamma, \beta)$ to determine the hyperparameters. We mark $x_0$ as feasible (yellow square) if problem~\eqref{eq:robustRelaxation} is feasible with some tuple of hyperparameters and infeasible (red dots) if~\eqref{eq:robustRelaxation} is verified as infeasible by Algorithm~\ref{alg:bisection}. Unverified $x_0$'s (white asterisk) are those for which a feasible tuple of hyperparameters have not been found during the grid search. It is observed that the feasibility region of robust SLS MPC is close to the exact region (green).

In Fig.~\ref{fig:tube_feasibility}, we evaluate the conservativeness of tube MPC with $Z_{\text{unit}}$ and $Z_{\text{inv}}$. Fig.~\ref{fig:Zunit_feasibility} shows that with $Z_{\text{unit}}$, tube MPC tends to be overly conservative for large horizons. With $Z_{\text{inv}}$, on the other hand, the feasibility region is close to the exact region (Fig.~\ref{fig:Zinv_feasibility}). Thus, the conservativeness of tube MPC largely depends on the choice of the tube parameterization, while robust SLS MPC does not have such a dependency.


 \begin{figure} [hbt!]
 	\centering
 	\begin{subfigure}{0.49 \columnwidth}\hfil
 		\includegraphics[width= \linewidth]{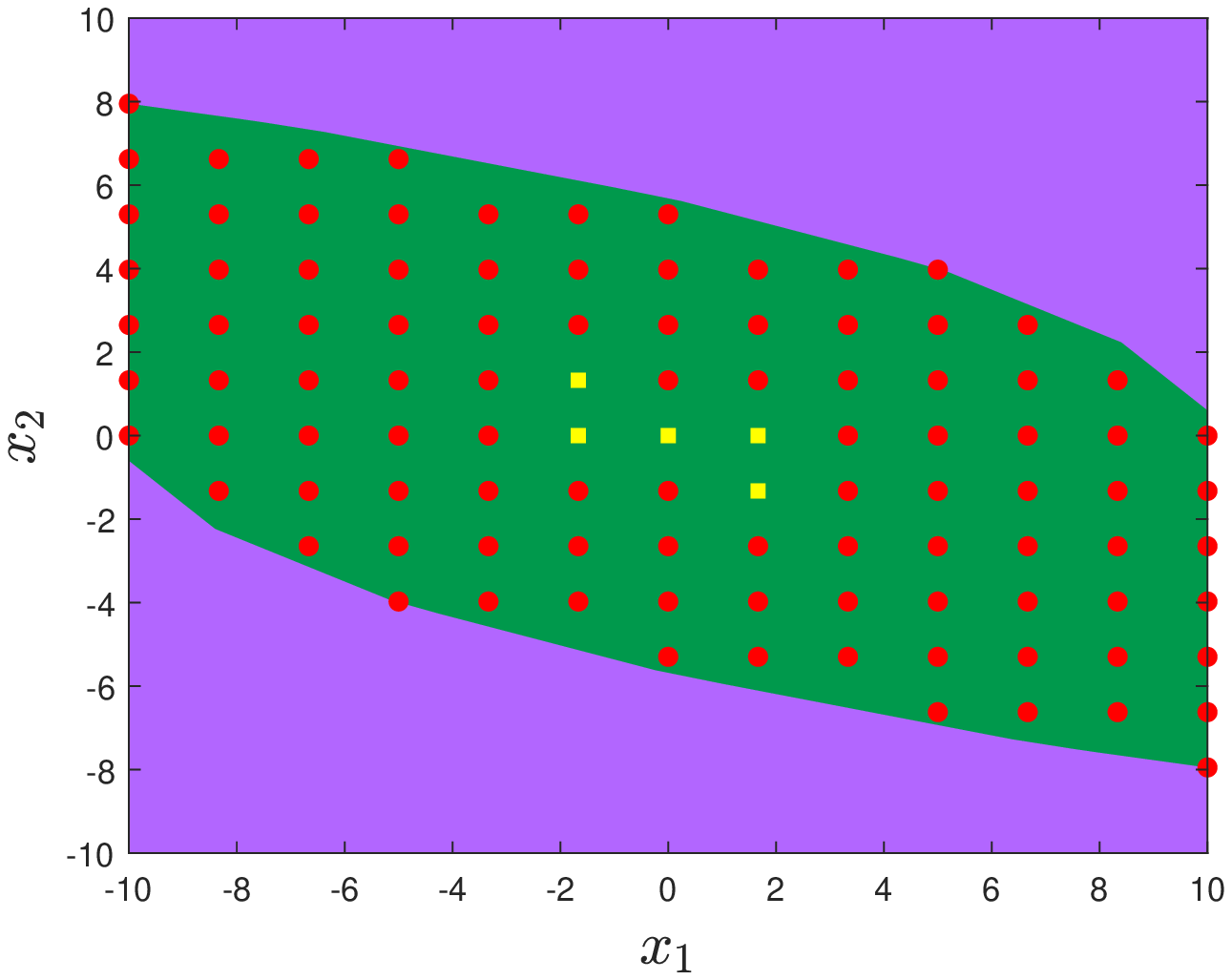}
 		\caption{Tube MPC with $Z_{\text{unit}}$.}
 		\label{fig:Zunit_feasibility}
 	\end{subfigure}
 	\begin{subfigure}{0.49 \columnwidth}
 		\includegraphics[width= \linewidth]{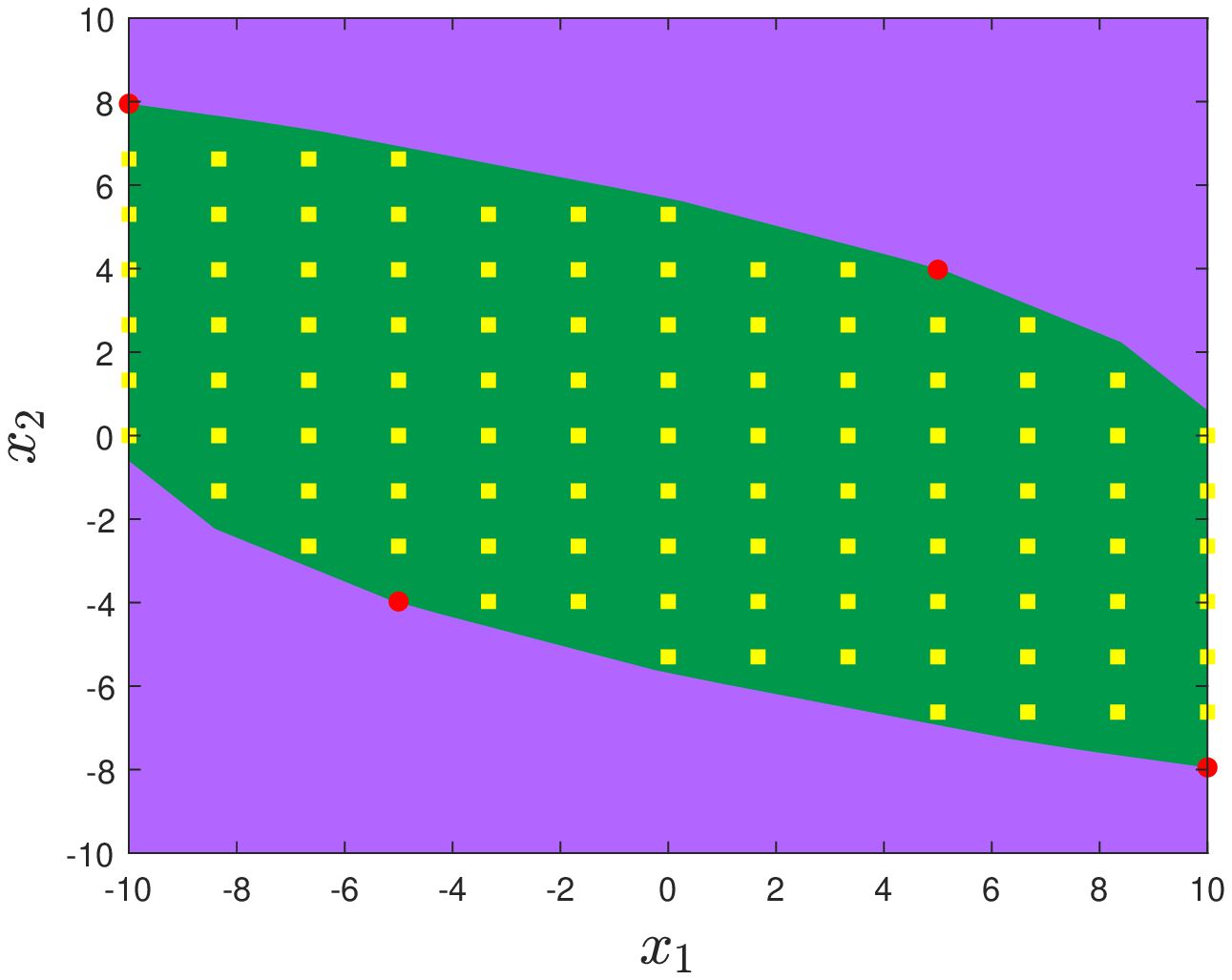}
 		\caption{Tube MPC with $Z_{\text{inv}}$}
 		\label{fig:Zinv_feasibility}
 	\end{subfigure}
 	\caption{Feasibility of the initial conditions in tube MPC with $Z_{\text{unit}}$ (left) and $Z_{\text{inv}}$ (right) for horizon $T = 10$. Yellow: feasible $x_0$'s. Red: infeasible $x_0$'s.}
 	\label{fig:tube_feasibility}
 \end{figure}
 
 \begin{figure} [hbt!]
	\centering
	\begin{subfigure}{0.49 \columnwidth}\hfil
		\includegraphics[width= \linewidth]{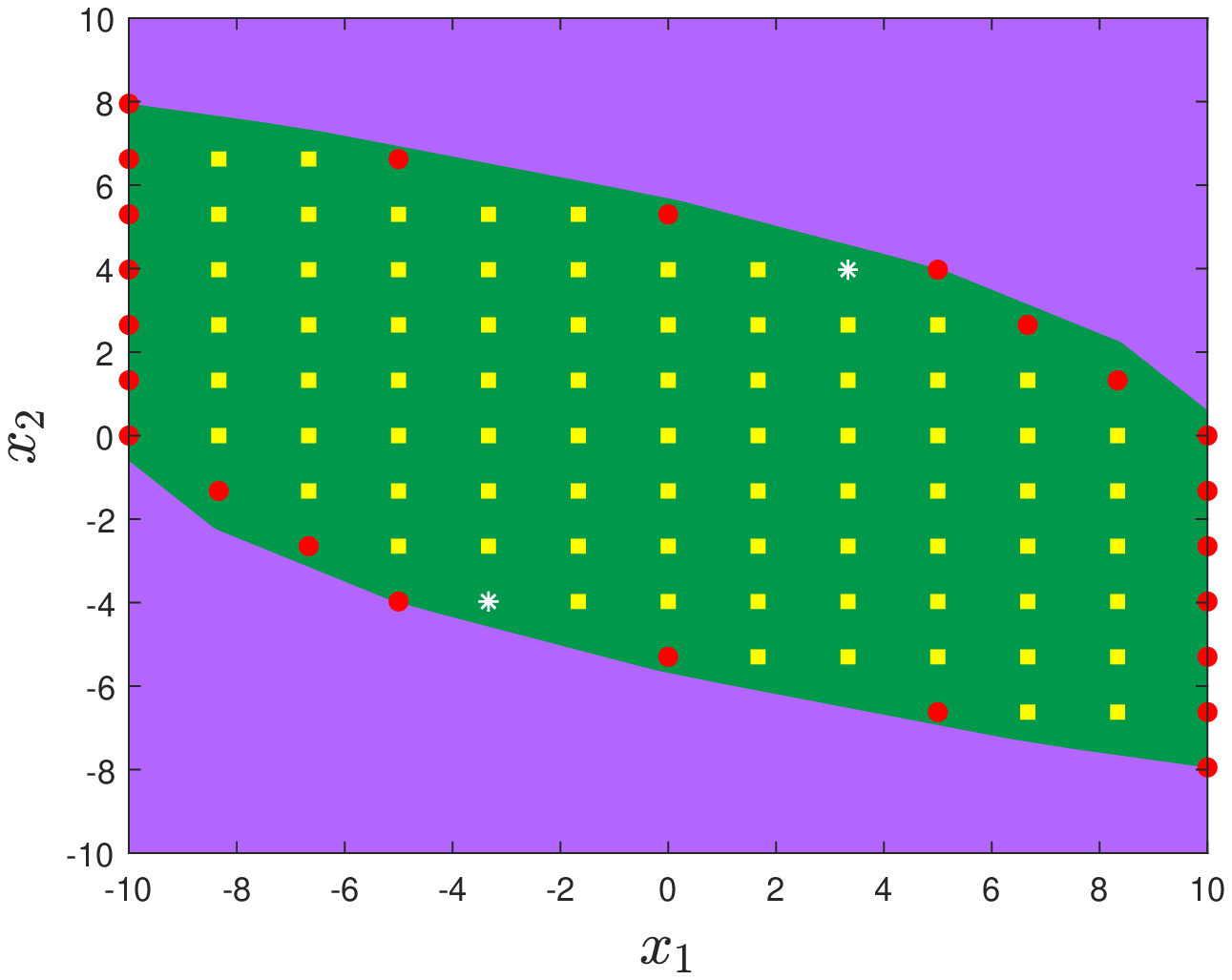}
		\caption{Robust SLS MPC.}
		\label{fig:slsfeasibility}
	\end{subfigure}
	\begin{subfigure}{0.49 \columnwidth}
		\includegraphics[width= \linewidth]{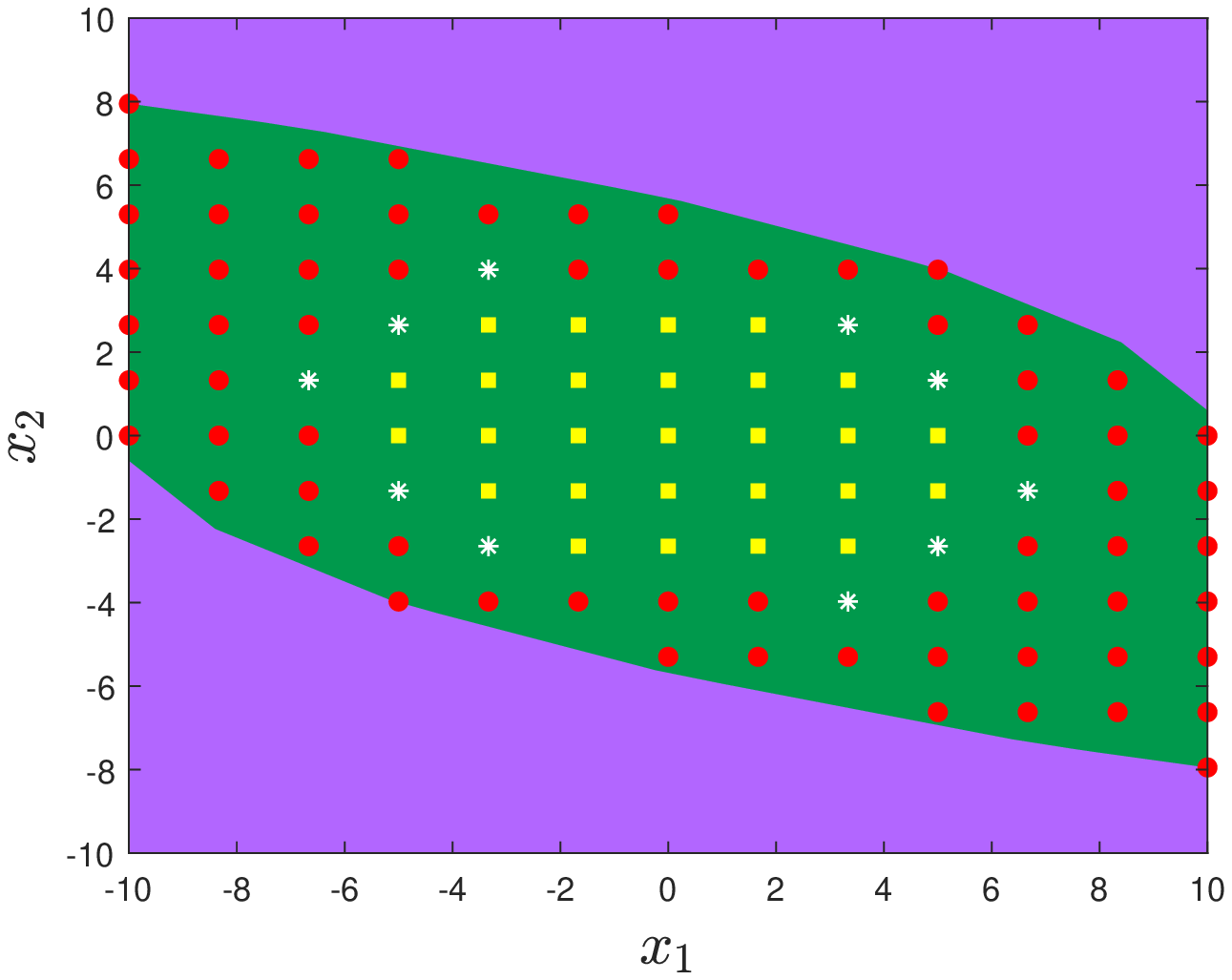}
		\caption{Coarse SLS MPC.}
		\label{fig:coarseslsfeasibility}
	\end{subfigure}
	\caption{Feasibility of the initial conditions in robust SLS MPC (left) and coarse SLS MPC (right) with horizon $T = 10$. Yellow: feasible $x_0$'s. Red: infeasible $x_0$'s. White: unverified $x_0$'s in the SLS relaxation.}
	\label{fig:coarse_sls_compare}
\end{figure}


\textbf{Comparison with existing SLS relaxations:} In the existing SLS literature~\cite{dean2019safely}, a coarser relaxation to that proposed in~\eqref{eq:robustRelaxation} has been used to solve robust OCPs, which we refer to as the ``coarse SLS'' approach. In this approach, there exists only one hyperparameter $\tau$, and we evaluate its conservativeness in Fig.~\ref{fig:coarseslsfeasibility} following the same procedure as in evaluating the SLS relaxation~\eqref{eq:robustRelaxation}. As shown, the SLS relaxation~\eqref{eq:robustRelaxation} in our paper considerably improves on the coarse SLS relaxation, while introducing only a small amount of additional computational overhead due to searching over the extra two hyperparameters $(\gamma, \beta)$, as described in Algorithm~\ref{alg:bisection}.

\section{Conclusion}
\label{sec:conclusion}


We proposed an SLS based approach for the robust MPC of LTV systems with norm bounded additive disturbances and LTV dynamic model uncertainty. A robust LTV state feedback controller is synthesized through SLS by optimizing over the closed-loop system responses. Computationally efficient convex relaxations of the robust optimal control problems are formulated. Simulation results indicate that the proposed framework achieves a favorable balance between conservativeness and computational complexity as compared with dynamic programming, tube MPC, and prior robust SLS based approaches.

\begin{appendices}
\section{}\label{appen:proof}

\begin{proof}[\textbf{Proof of Theorem}~\ref{thm:robustPerformance}]
	Recall that $\DDelta_c = \{ \DDelta \in \mathcal{L}_{TV} \mid \lVert \DDelta \rVert_\indinf \leq 1, \DDelta \text{ is causal}  \}$. Consider matrices
	\begin{equation} \label{eq:tildedelta}
	\MM = \begin{bmatrix}
	\XX & \XX \\ \YY & \YY
	\end{bmatrix},\quad 
	\widetilde{\DDelta} = \begin{bmatrix}
	\DDelta_1 & \\ & \DDelta_2
	\end{bmatrix}
	\end{equation}
	where $\XX \in \mathbb{R}^{p_1 \times q}, \YY \in \mathbb{R}^{p_2 \times q}, \DDelta_1 \in \mathbb{R}^{q \times p_1}, \DDelta_2 \in \mathbb{R}^{q \times p_2}, \DDelta_1, \DDelta_2 \in \DDelta_c$. Next we want to show that 
	\begin{equation*}
	\lVert \XX \rVert_\indinf + \lVert \YY \rVert_\indinf < 1 \Rightarrow (I - \MM\widetilde{\DDelta})^{-1} \ \text{exists} \ \forall \widetilde{\DDelta}.
	\end{equation*} 
	For $\forall d_1, d_2 > 0$, define 
	\begin{equation*}
	\begin{aligned}
	& D_p = \begin{bmatrix}
	d_1 I_{p_1} & \\ & d_2 I_{p_2}
	\end{bmatrix}, \quad
	D_p^{-1} = \begin{bmatrix}
	\frac{1}{d_1} I_{p_1} & \\ & \frac{1}{d_2} I_{p_2}
	\end{bmatrix}, \\
	& D_q = \begin{bmatrix}
	d_1 I_{q} & \\ & d_2 I_{q}
	\end{bmatrix}, \quad
	D_q^{-1} = \begin{bmatrix}
	\frac{1}{d_1} I_{q} & \\ & \frac{1}{d_2} I_{q}
	\end{bmatrix}.
	\end{aligned}
	\end{equation*}
	It follows that $D_q^{-1} \widetilde{\DDelta} = \widetilde{\DDelta} D_p^{-1}$. Since $D_p(I - \MM \widetilde{\DDelta})D_p^{-1} = I - D_p \MM \widetilde{\DDelta} D_p^{-1} = I - D_p \MM D_q^{-1} \widetilde{\DDelta}$, we have $I - \MM \widetilde{\DDelta}$ is invertible if and only if $I - D_p \MM D_q^{-1} \widetilde{\DDelta}$ is invertible for some $d_1, d_2 > 0$. Substitute $D_p, D_q^{-1}$ and we have 
	\begin{equation*}
	D_p \MM D_q^{-1} =  \begin{bmatrix}
	\XX & \frac{d_1}{d_2} \XX \\ \frac{d_2}{d_1} \YY & \YY 
	\end{bmatrix} = 
	\begin{bmatrix}
	\XX & \lambda \XX \\ \frac{1}{\lambda } \YY & \YY
	\end{bmatrix}.
	\end{equation*}
	with $\lambda = \frac{d_1}{d_2} > 0$. Then $\lVert D_p \MM D_q^{-1} \rVert_\indinf \leq \max\{\lVert \XX \rVert_\indinf + \lambda \lVert \XX \rVert_\indinf, \frac{1}{\lambda} \lVert \YY \rVert_\indinf + \lVert \YY \rVert_\indinf \}$. Taking the infimum over $\lambda$ (or equivalently $d_1, d_2$), we have
	\begin{equation*}
	\begin{aligned}
	&\quad \ \inf_{\lambda} \ \lVert D_p \MM D_q^{-1} \rVert_\indinf \\
	&\leq \inf_\lambda \max \lbrace \lVert \XX \rVert_\indinf + \lambda \lVert \XX \rVert_\indinf, \\
	& \qquad \qquad \qquad \frac{1}{\lambda} \lVert \YY \rVert_\indinf + \lVert \YY \rVert_\indinf \rbrace \\
	& = \lVert \XX \rVert_\indinf + \lVert \YY \rVert_\indinf,
	\end{aligned}
	\end{equation*}
	where the infimum is achieved by $\lambda = d_1 /d_2 = \lVert \YY \rVert_\indinf / \lVert \XX \rVert_\indinf$. Thus 
	\begin{equation*}
	\begin{aligned}
	& \quad \ \lVert \XX \rVert_\indinf + \lVert \YY \rVert_\indinf < 1 \\
	& \Rightarrow \inf_{d_1, d_2} \ \lVert D_p \MM D_q^{-1} \rVert_\indinf < 1 \\
	& \Rightarrow (I - \MM\widetilde{\DDelta})^{-1} \ \text{exists for all } \widetilde{\DDelta} \text{ in \eqref{eq:tildedelta}}
	\end{aligned}
	\end{equation*}
	by the small gain theorem~\cite{dullerud2013course}. 
	By proposition~\ref{prop:equivalence}, we have 
	\begin{equation} \label{eq:sufficientRobustPerf}
	\begin{aligned}
	& \quad \ \lVert \XX \rVert_\indinf + \lVert \YY \rVert_\indinf < 1 \\
	& \Rightarrow \lVert \XX + \XX \DDelta_1(I - \YY \DDelta_1)^{-1} \YY  \rVert_\indinf < 1, \forall \DDelta_1 \in \DDelta_c.
	\end{aligned}
	\end{equation}
	Let
	\begin{equation*}
	\MM = \begin{bmatrix}
	\MM_{11} & \MM_{12} \\ \MM_{21} & \MM_{22}
	\end{bmatrix} = 
	\begin{bmatrix}
	\frac{1}{\beta} \mathbf{P} \hat{\PPhi} & \frac{1}{\beta} \mathbf{P} \hat{\PPhi} \\ \epsilon \hat{\PPhi} & \epsilon \hat{\PPhi}
	\end{bmatrix}.
	\end{equation*}
	Apply~\eqref{eq:sufficientRobustPerf} and we prove the theorem.
\end{proof}

\begin{proof}[\textbf{Proof of Theorem}~\ref{thm:robustConvexRelaxation}]
Since $Z$ is a block-downshift operator, from the definitions in~\eqref{eq:defPhiDelta} and~\eqref{eq:matStructure}, we have $\DeltaVar^0 = 0$ and $\PPhi \DeltaVar = \PPhi^0 \DeltaVar^0 + \PPhi^\tildeww \DeltaVar^\tildeww = \PPhi^\tildeww \DeltaVar^\tildeww$. Then the robust constraint in~\eqref{eq:robustMPCprob} can be written as 
\begin{equation} \label{eq:robustConstr}
\begin{aligned}
F_j \begin{bmatrix}
\xx \\ \uu
\end{bmatrix} & = F_j^\tp \PPhi \ww + F_j^\tp \PPhi \DeltaVar( I - \PPhi \DeltaVar)^{-1} \PPhi \ww \\
& = F_j^\tp \PPhi^0 x_0 + F_j^\tp \PPhi^\tildeww \DeltaVar^\tildeww (I - \PPhi^\tildeww \DeltaVar^\tildeww)^{-1} \PPhi^0 x_0 \\
& \quad + F_j^\tp(\PPhi^\tildeww + \PPhi^\tildeww \DeltaVar^\tildeww(I - \PPhi^\tildeww \DeltaVar^\tildeww)^{-1} \PPhi^\tildeww) \tildeww \\
& \leq b_j
\end{aligned}
\end{equation}
for all $j$ indexing the $j$-th row of $F$ and $b$ and for $\forall w_t \in \mathcal{W},  \forall \lVert {\Delta_A}_t \rVert_\indinf \leq \epsilon_A, \forall \lVert {\Delta_B}_t \rVert_\indinf \leq \epsilon_B$. Next, we upperbound the term $\circled{1} = F_j^\tp(\PPhi^\tildeww + \PPhi^\tildeww \DeltaVar^\tildeww(I - \PPhi^\tildeww \DeltaVar^\tildeww)^{-1} \PPhi^\tildeww) \tildeww$ and the term $\circled{2} =  F_j^\tp \PPhi^\tildeww \DeltaVar^\tildeww (I - \PPhi^\tildeww \DeltaVar^\tildeww)^{-1} \PPhi^0 x_0$ separately. 

\textbf{Bounding $\circled{1}$}: By the submultiplicativity of the matrix induced infinity-norm, we can upperbound term $\circled{1}$ by 
\begin{equation}
\begin{aligned}
& \quad F_j^\tp(\PPhi^\tildeww + \PPhi^\tildeww \DeltaVar^\tildeww(I - \PPhi^\tildeww \DeltaVar^\tildeww)^{-1} \PPhi^\tildeww) \tildeww  \\
& \leq \lVert F_j^\tp(\PPhi^\tildeww + \PPhi^\tildeww \DeltaVar^\tildeww(I - \PPhi^\tildeww \DeltaVar^\tildeww)^{-1} \PPhi^\tildeww) \tildeww \rVert_\infty \\
& \leq \lVert F_j^\tp(\PPhi^\tildeww + \PPhi^\tildeww \DeltaVar^\tildeww(I - \PPhi^\tildeww \DeltaVar^\tildeww)^{-1} \PPhi^\tildeww) \rVert_\indinf \sigma_w. 
\end{aligned}
\end{equation}
It can be easily checked that $\lVert \DeltaVar^\tildeww \rVert_\indinf \leq \epsA + \epsB = \epsilon$ and $(I - \PPhi^\tildeww \DeltaVar^\tildeww)^{-1}$ exists. By Theorem~\ref{thm:robustPerformance}, the following two statements are equivalent:
\begin{equation}
\begin{aligned}
& (i) \ \lVert F_j^\tp(\PPhi^\tildeww + \PPhi^\tildeww \DeltaVar^\tildeww(I - \PPhi^\tildeww \DeltaVar^\tildeww)^{-1} \PPhi^\tildeww) \rVert_\indinf \leq \beta_j. \\
& (ii) \lVert F_j^\tp \PPhi^\tildeww \rVert_\indinf + \beta_j \lVert \epsilon \PPhi^\tildeww \rVert_\indinf \leq \beta_j.
\end{aligned}
\end{equation}
Then $\circled{1} \leq \beta_j \sigma_w$ if $\lVert F_j^\tp \PPhi^\tildeww \rVert_\indinf + \beta_j \lVert \epsilon \PPhi^\tildeww \rVert_\indinf \leq \beta_j$. To reduce the number of hyperparameters, we replace $\beta_j$ by a uniform $\beta$ for all $j$. This still guarantees the robustness of the constraints in~\eqref{eq:robustRelaxation} and can be seen as a tradeoff between computational complexity and conservativeness. 

\textbf{Bounding $\circled{2}$}:
Since $\DeltaVar = \left[ \DeltaVar_A \mid \DeltaVar_B \right]$ and $\DeltaVar_A, \DeltaVar_B$ are strict block-lower-triangular matrices, we have $(\PPhi^\tildeww \DeltaVar^\tildeww)^{T+1} = 0$ and $(I - \PPhi^\tildeww \DeltaVar^\tildeww)^{-1} = \sum_{k = 0}^T (\PPhi^\tildeww \DeltaVar^\tildeww)^k$. Then we can bound term $\circled{2}$ as:
\begin{align} \label{eq:term2}
& \quad F_j^\tp \PPhi^\tildeww \DeltaVar^\tildeww (I - \PPhi^\tildeww \DeltaVar^\tildeww)^{-1} \PPhi^0 x_0 \nonumber \\
& = F_j^\tp \sum_{k = 1}^T (\PPhi^\tildeww \DeltaVar^\tildeww)^k \PPhi^0 x_0 \nonumber \\
& = F_j^\tp \PPhi^\tildeww \sum_{k = 0}^{T-1} ( \DeltaVar^\tildeww \PPhi^\tildeww)^k \DeltaVar^\tildeww \PPhi^0 x_0 \nonumber\\
& \leq \lVert F_j^\tp \PPhi^\tildeww\rVert_1 \lVert     \sum_{k = 0}^{T-1} (\DeltaVar^\tildeww \PPhi^\tildeww)^k \DeltaVar^\tildeww \PPhi^0 x_0 \rVert_\infty \nonumber\\
& \leq \lVert F_j^\tp \PPhi^\tildeww\rVert_1 \lVert     \sum_{k = 0}^{T-1} (\DeltaVar^\tildeww \PPhi^\tildeww)^k \rVert_\indinf \lVert \DeltaVar^\tildeww \PPhi^0 x_0 \rVert_\infty \nonumber\\
& \leq \lVert F_j^\tp \PPhi^\tildeww\rVert_1      \sum_{k = 0}^{T-1} \lVert \DeltaVar^\tildeww \PPhi^\tildeww\rVert_\indinf^k  \lVert \DeltaVar^\tildeww \PPhi^0 x_0 \rVert_\infty \nonumber\\
& \leq \lVert F_j^\tp \PPhi^\tildeww\rVert_1 \frac{1 - \tau^T}{1 - \tau} \gamma
\end{align}
for any $\tau, \gamma > 0$ such that $\lVert \DeltaVar^\tildeww \PPhi^\tildeww  \rVert_\indinf \leq \tau$ and $\lVert \DeltaVar^\tildeww \PPhi^0 x_0 \rVert_\infty \leq \gamma$. The first inequality applies Holder's inequality while the rest inequalities are from the triangular inequality and submultiplicativity of the induced infinity-norm.

\textbf{Bounds with $\tau, \gamma$}: 
First we bound $\lVert \DDelta^\tildeww \PPhi^\tildeww  \rVert_\indinf $ by:
\begin{equation}
\begin{aligned}
& \quad \ \lVert \DeltaVar^\tildeww \PPhi^\tildeww \rVert_\indinf \\
& = \Big\lVert \begin{bmatrix} \DeltaVar^\tildeww_A \frac{\alpha}{\epsilon_A} & \DeltaVar^\tildeww_B \frac{1 - \alpha}{\epsilon_B} \end{bmatrix} \begin{bmatrix} \frac{\epsilon_A}{\alpha} \Phix^\tildeww \\ \frac{\epsilon_B}{1 - \alpha}\Phiu^\tildeww \end{bmatrix}  \Big\rVert_\indinf  \\
& \leq \Big \lVert \begin{bmatrix} \DeltaVar^\tildeww_A \frac{\alpha}{\epsilon_A} & \DeltaVar^\tildeww_B \frac{1 - \alpha}{\epsilon_B} \end{bmatrix} \Big \rVert_\indinf \Big \lVert \begin{bmatrix} \frac{\epsilon_A}{\alpha} \Phix^\tildeww \\ \frac{\epsilon_B}{1 - \alpha}\Phiu^\tildeww \end{bmatrix} \Big \rVert_{ \indinf} \\
& \leq (\frac{\alpha}{\epsilon_A} \lVert \DeltaVar_A^\tildeww \rVert_\indinf + \frac{1 - \alpha}{\epsilon_B} \lVert \DeltaVar_B^\tildeww \rVert_\indinf) \Big \lVert \begin{bmatrix} \frac{\epsilon_A}{\alpha} \Phix^\tildeww \\ \frac{\epsilon_B}{1 - \alpha}\Phiu^\tildeww \end{bmatrix} \Big \rVert_{ \indinf} \\
& \leq \Big \lVert \begin{bmatrix} \frac{\epsilon_A}{\alpha} \Phix^\tildeww \\ \frac{\epsilon_B}{1 - \alpha}\Phiu^\tildeww \end{bmatrix}  \Big\rVert_\indinf
\end{aligned}
\end{equation}
for any $\alpha \in (0,1)$. The first inequality follows from submultiplicativity and the second inequality holds because the induced infinity-norm is the maximum absolute row sum a matrix. For $\lVert \DDelta^\tildeww \PPhi^0 x_0 \rVert_\infty$ a similar upperbound can be obtained. Then the following conditions
\begin{equation}
\Big\lVert \begin{bmatrix} \frac{\epsilon_A}{\alpha} \Phix^\tildeww \\  \frac{\epsilon_B}{1 - \alpha} \Phiu^\tildeww \end{bmatrix} \Big\rVert_\indinf \leq \tau, \quad 
\Big\lVert \begin{bmatrix} \frac{\epsilon_A}{\alpha} \Phix^0 \\  \frac{\epsilon_B}{1 - \alpha} \Phiu^0 \end{bmatrix} x_0 \Big\rVert_\infty \leq \gamma
\end{equation}
guarantee that $\lVert \DDelta^\tildeww \PPhi^\tildeww  \rVert_\indinf \leq \tau$ and $\lVert \DDelta^\tildeww \PPhi^0 x_0 \rVert_\infty \leq \gamma $ hold for all $\lVert {\Delta_A}_t \rVert_\indinf \leq \epsA, \lVert {\Delta_B}_t \rVert_\indinf \leq \epsB$. Putting everything together, we prove Theorem~\ref{thm:robustConvexRelaxation}.
\end{proof}

\end{appendices}

\bibliographystyle{IEEEtran}
\bibliography{reference}

\begin{thebibliography}{10}
\providecommand{\url}[1]{#1}
\csname url@samestyle\endcsname
\providecommand{\newblock}{\relax}
\providecommand{\bibinfo}[2]{#2}
\providecommand{\BIBentrySTDinterwordspacing}{\spaceskip=0pt\relax}
\providecommand{\BIBentryALTinterwordstretchfactor}{4}
\providecommand{\BIBentryALTinterwordspacing}{\spaceskip=\fontdimen2\font plus
\BIBentryALTinterwordstretchfactor\fontdimen3\font minus
  \fontdimen4\font\relax}
\providecommand{\BIBforeignlanguage}[2]{{%
\expandafter\ifx\csname l@#1\endcsname\relax
\typeout{** WARNING: IEEEtran.bst: No hyphenation pattern has been}%
\typeout{** loaded for the language `#1'. Using the pattern for}%
\typeout{** the default language instead.}%
\else
\language=\csname l@#1\endcsname
\fi
#2}}
\providecommand{\BIBdecl}{\relax}
\BIBdecl

\bibitem{qin2003survey}
S.~J. Qin and T.~A. Badgwell, ``A survey of industrial model predictive control
  technology,'' \emph{Control engineering practice}, vol.~11, no.~7, pp.
  733--764, 2003.

\bibitem{hrovat2012development}
D.~Hrovat, S.~Di~Cairano, H.~E. Tseng, and I.~V. Kolmanovsky, ``The development
  of model predictive control in automotive industry: A survey,'' in \emph{2012
  IEEE International Conference on Control Applications}.\hskip 1em plus 0.5em
  minus 0.4em\relax IEEE, 2012, pp. 295--302.

\bibitem{negenborn2007multi}
R.~R. Negenborn, ``Multi-agent model predictive control with applications to
  power networks,'' 2007.

\bibitem{wieber2006trajectory}
P.-B. Wieber, ``Trajectory free linear model predictive control for stable
  walking in the presence of strong perturbations,'' in \emph{2006 6th IEEE-RAS
  International Conference on Humanoid Robots}.\hskip 1em plus 0.5em minus
  0.4em\relax IEEE, 2006, pp. 137--142.

\bibitem{marruedo2002stability}
D.~L. Marruedo, T.~Alamo, and E.~Camacho, ``Stability analysis of systems with
  bounded additive uncertainties based on invariant sets: Stability and
  feasibility of mpc,'' in \emph{Proceedings of the 2002 American Control
  Conference (IEEE Cat. No. CH37301)}, vol.~1.\hskip 1em plus 0.5em minus
  0.4em\relax IEEE, 2002, pp. 364--369.

\bibitem{grimm2004examples}
G.~Grimm, M.~J. Messina, S.~E. Tuna, and A.~R. Teel, ``Examples when nonlinear
  model predictive control is nonrobust,'' \emph{Automatica}, vol.~40, no.~10,
  pp. 1729--1738, 2004.

\bibitem{bemporad1999robust}
A.~Bemporad and M.~Morari, ``Robust model predictive control: A survey,'' in
  \emph{Robustness in identification and control}.\hskip 1em plus 0.5em minus
  0.4em\relax Springer, 1999, pp. 207--226.

\bibitem{mayne2000constrained}
D.~Q. Mayne, J.~B. Rawlings, C.~V. Rao, and P.~O. Scokaert, ``Constrained model
  predictive control: Stability and optimality,'' \emph{Automatica}, vol.~36,
  no.~6, pp. 789--814, 2000.

\bibitem{kouvaritakis2000efficient}
B.~Kouvaritakis, J.~A. Rossiter, and J.~Schuurmans, ``Efficient robust
  predictive control,'' \emph{IEEE Transactions on automatic control}, vol.~45,
  no.~8, pp. 1545--1549, 2000.

\bibitem{schuurmans2000robust}
J.~Schuurmans and J.~Rossiter, ``Robust predictive control using tight sets of
  predicted states,'' \emph{IEE proceedings-Control theory and applications},
  vol. 147, no.~1, pp. 13--18, 2000.

\bibitem{lee2000robust}
Y.~I. Lee and B.~Kouvaritakis, ``Robust receding horizon predictive control for
  systems with uncertain dynamics and input saturation,'' \emph{Automatica},
  vol.~36, no.~10, pp. 1497--1504, 2000.

\bibitem{goulart2006optimization}
P.~J. Goulart, E.~C. Kerrigan, and J.~M. Maciejowski, ``Optimization over state
  feedback policies for robust control with constraints,'' \emph{Automatica},
  vol.~42, no.~4, pp. 523--533, 2006.

\bibitem{mayne2005robust}
D.~Q. Mayne, M.~M. Seron, and S.~Rakovi{\'c}, ``Robust model predictive control
  of constrained linear systems with bounded disturbances,'' \emph{Automatica},
  vol.~41, no.~2, pp. 219--224, 2005.

\bibitem{langson2004robust}
W.~Langson, I.~Chryssochoos, S.~Rakovi{\'c}, and D.~Q. Mayne, ``Robust model
  predictive control using tubes,'' \emph{Automatica}, vol.~40, no.~1, pp.
  125--133, 2004.

\bibitem{kothare1996robust}
M.~V. Kothare, V.~Balakrishnan, and M.~Morari, ``Robust constrained model
  predictive control using linear matrix inequalities,'' \emph{Automatica},
  vol.~32, no.~10, pp. 1361--1379, 1996.

\bibitem{borrelli2017predictive}
F.~Borrelli, A.~Bemporad, and M.~Morari, \emph{Predictive control for linear
  and hybrid systems}.\hskip 1em plus 0.5em minus 0.4em\relax Cambridge
  University Press, 2017.

\bibitem{bujarbaruah2019semi}
M.~Bujarbaruah, S.~H. Nair, and F.~Borrelli, ``A semi-definite programming
  approach to robust adaptive mpc under state dependent uncertainty,''
  \emph{arXiv preprint arXiv:1910.04378}, 2019.

\bibitem{kim2018robust}
Y.~Kim, X.~Zhang, J.~Guanetti, and F.~Borrelli, ``Robust model predictive
  control with adjustable uncertainty sets,'' in \emph{2018 IEEE Conference on
  Decision and Control (CDC)}.\hskip 1em plus 0.5em minus 0.4em\relax IEEE,
  2018, pp. 5176--5181.

\bibitem{lorenzen2019robust}
M.~Lorenzen, M.~Cannon, and F.~Allg{\"o}wer, ``Robust mpc with recursive model
  update,'' \emph{Automatica}, vol. 103, pp. 461--471, 2019.

\bibitem{ANDERSON2019364}
J.~Anderson, J.~C. Doyle, S.~H. Low, and N.~Matni, ``System level synthesis,''
  \emph{Annual Reviews in Control}, vol.~47, pp. 364 -- 393, 2019.

\bibitem{dean2019safely}
S.~Dean, S.~Tu, N.~Matni, and B.~Recht, ``Safely learning to control the
  constrained linear quadratic regulator,'' in \emph{2019 American Control
  Conference (ACC)}.\hskip 1em plus 0.5em minus 0.4em\relax IEEE, 2019, pp.
  5582--5588.

\bibitem{horn2012matrix}
R.~A. Horn and C.~R. Johnson, \emph{Matrix analysis}.\hskip 1em plus 0.5em
  minus 0.4em\relax Cambridge university press, 2012.

\bibitem{Dahleh1994ControlOU}
M.~A. Dahleh and I.~J. Diaz-Bobillo, ``Control of uncertain systems: a linear
  programming approach,'' 1994.

\bibitem{khammash1991performance}
M.~Khammash and J.~Pearson, ``Performance robustness of discrete-time systems
  with structured uncertainty,'' \emph{IEEE Transactions on Automatic Control},
  vol.~36, no.~4, pp. 398--412, 1991.

\bibitem{dullerud2013course}
G.~E. Dullerud and F.~Paganini, \emph{A course in robust control theory: a
  convex approach}.\hskip 1em plus 0.5em minus 0.4em\relax Springer Science \&
  Business Media, 2013, vol.~36.

\bibitem{lofberg2004yalmip}
J.~L{\"o}fberg, ``Yalmip: A toolbox for modeling and optimization in matlab,''
  in \emph{Proceedings of the CACSD Conference}, vol.~3.\hskip 1em plus 0.5em
  minus 0.4em\relax Taipei, Taiwan, 2004.

\bibitem{andersen2000mosek}
E.~D. Andersen and K.~D. Andersen, ``The mosek interior point optimizer for
  linear programming: an implementation of the homogeneous algorithm,'' in
  \emph{High performance optimization}.\hskip 1em plus 0.5em minus 0.4em\relax
  Springer, 2000, pp. 197--232.

\bibitem{wang2009fast}
Y.~Wang and S.~Boyd, ``Fast model predictive control using online
  optimization,'' \emph{IEEE Transactions on control systems technology},
  vol.~18, no.~2, pp. 267--278, 2009.

\end{thebibliography}

\end{document}